\documentclass{article}

\usepackage{booktabs, multirow, bm}
\usepackage{array}      
\usepackage{amsmath, amssymb, amsthm, mathrsfs}
\usepackage{indentfirst}
\usepackage{graphicx, tikz, subfigure}
\usepackage{geometry}
\usepackage{float}
\usepackage{longtable}
\usepackage{caption}
\usepackage{multirow}
\numberwithin{equation}{section}
\usepackage{xcolor}
\usepackage{ulem}
\usepackage{cite}
\usepackage{hyperref}
\allowdisplaybreaks[4]
\usepackage{cleveref}
\crefname{section}{Section}{Sections}
\crefname{figure}{Figure}{Figures}
\crefname{table}{Table}{Tables}
\crefname{equation}{}{}
\crefname{theorem}{Theorem}{Theorems}
\crefname{lemma}{Lemma}{Lemmas}
\crefname{remark}{Remark}{Remarks}
\crefname{problem}{Problem}{Subproblems}

\newtheorem{theorem}{Theorem}[section]

\newtheorem{remark}{Remark}[section]
\newtheorem{lemma}{Lemma}[section]

\newtheorem{definition}{Definition}[section]
\newtheorem{Assumption}{Assumption}
\theoremstyle{definition}

\newtheorem{example}{\noindent Example}[subsection]

\crefname{ip}{Co-inversion Problem}{ips}
\newtheoremstyle{MyThmStyle}
{}
{}
{}
{}
{\bfseries}
{}
{ }
{\thmname{#1\thmnumber{ #2\hspace{0.5em}}}\thmnote{(#3)}}
\theoremstyle{MyThmStyle}

\crefname{subisp}{inverse source problem}{ips}
\crefname{subiop}{inverse obstacle problem}{ips}
\graphicspath{{figurenew/}}
\definecolor{bananamania}{rgb}{0.98, 0.91, 0.71}

\begin{document}
	
	\title{Optimal Estimation and Uncertainty Quantification for Stochastic Inverse Problems via Variational Bayesian Methods}
	\author{
		Ruibiao Song\thanks{College of science, China University of Mining and Technology (Beijing), Beijing, China. {\it 15808630600@163.com}},
		Liying Zhang\thanks{College of science, China University of Mining and Technology (Beijing), Beijing, China. {\it lyzhang@lsec.cc.ac.cn}}
	}
	\date{}
	\maketitle
	\begin{abstract}
		The Bayesian inversion method demonstrates significant potential for solving inverse problems, enabling both point estimation and uncertainty quantification (UQ). However, Bayesian maximum a posteriori (MAP) estimation may become unstable when handling data from diverse distributions (e.g., solutions of stochastic partial differential equations (SPDEs)). Additionally, Monte Carlo sampling methods are computationally expensive. To address these challenges, we propose a novel two-stage optimization method based on optimal control theory and variational Bayesian methods. This method not only yields stable solutions for stochastic inverse problems but also efficiently quantifies the uncertainty of solutions. In the first stage, we introduce a new weighting formulation to ensure the stability of the Bayesian MAP estimation. In the second stage, we derive the necessary condition for efficiently quantifying the uncertainty of the solutions by combining the new weighting formula with variational inference. Furthermore, we establish an error estimation theorem that relates the exact solution to the optimally estimated solution under different amounts of observed data. Finally, the efficiency of the proposed method is demonstrated through numerical examples.
	\end{abstract}
	
	\section{Introduction}
	The inverse problem of stochastic partial differential equations (SPDEs) generally involves qualitatively analyzing and quantitatively determining unknown parameters based on observed data. Since the solutions of the SPDEs follow specific probability distributions, existing works on stochastic inverse problems primarily focus on the statistical characteristics of the observed data (e.g., expectation and variance). For example, Bao et al. developed a fast Fourier transform strategy based on these statistical properties to reconstruct the random source function in one-dimensional stochastic Helmholtz equations (\cite{ref11}). Li et al. proved the uniqueness of the random source term in stochastic wave equations by utilizing expectation and covariance information from terminal time data (\cite{ref25}). Moreover, the inverse problems related to parabolic SPDEs can also be explored through statistical characteristics. Ruan et al. established the uniqueness of time-varying coefficient identification based on the expectation of the observed data (\cite{ref35}). Niu et al. proved the uniqueness of the expectation and variance of the source term in stochastic fractional diffusion equations, relying on the statistics of the random source term (\cite{ref26}). Feng et al. also demonstrated the uniqueness of the potential function in stochastic diffusion equations affected by multiplicative white noise using statistical information from the observed data (\cite{ref31}). In numerical solutions, regularization techniques are employed to overcome the ill-posedness of stochastic inverse problems, similar to the deterministic case. For instance, Bao et al. introduced a regularized Kaczmarz method for reconstructing the random source in stochastic elastic wave equations (\cite{ref12}). Niu et al. used the Tikhonov regularization method to derive the regularized solution of stochastic fractional diffusion equations (\cite{ref26}). Similarly, Dou et al. applied the Tikhonov regularization method to study the inverse of initial value problems for parabolic SPDEs (\cite{ref10}). Recently, some scholars have used intelligent algorithms to study stochastic inverse problems (\cite{ref2, ref5, ref15, ref32}).
	\par
	Besides, statistical inversion methods are becoming popular for solving inverse problems. In particular, the Bayesian inversion method has demonstrated unique advantages in studying deterministic inverse problems. It not only provides the best point estimates of solutions but also provides the UQ of solutions. The Bayesian inversion method has been successfully applied to various inverse problems (\cite{  ref1, ref3, ref7, ref17, ref27, ref29, ref33}). However, the Bayesian inversion method becomes unstable when directly applied to solve stochastic inverse problems. This is because the solutions of SPDEs typically follow complex probability distributions, which can lead to an ill-conditioned covariance matrix for the observed data (for more details, see Section 2). There are few studies on the Bayesian inversion method for solving stochastic inverse problems (\cite{ref18,ref30}). Additionally, the Bayesian inversion method is computationally expensive for the UQ of solutions (\cite{ref16,ref18}). Therefore, in this paper, we consider the following questions:
	\begin{itemize}
		\item Q1: How to make the Bayesian inversion method stable while solving the stochastic inverse problems?
		\item Q2: How to construct a numerical method to quantify the uncertainty of the solution efficiently? 
		\item Q3: How do we estimate the error between the exact solution and the inversion solution under different amounts of observed data, thereby quantifying the uncertainty of the solution theoretically?
	\end{itemize}
	\par
	The three questions above are the motivations of this work, and we give detailed answers in the following sections of this paper. Next, we take the source term reconstruction in the parabolic SPDEs (1.1) as an example to elaborate on the key difficulties and main methods.
	\begin{equation} 
		\left\{\begin{aligned}
			du(x, t) & =Au(x,t) dt+R(t)f(x)dt+g(x)dw(t), &&(x, t)\in D\times(0, T], \\
			u(x, t) & =0, &&(x,t)\in\partial D \times(0, T], \\
			u(x, t) & =u_{0}(x), && x\in \bar{D}, t=0.
		\end{aligned}\right.
		\tag{1.1}
		\label{eq:1.1} 
	\end{equation} 
	\par
	Let $D \subset \mathbb{R}^d(d=1,2,3)$ be a bounded domain with a Lipschitz boundary, and $\bar{D} = D \cup \partial D$. The operator $A$ is a symmetric uniformly elliptic operator with the homogeneous Dirichlet boundary condition. It is defined as for $u \in H^{2}(D) \cap H_{0}^{1}(D)$,
	\[
	Au(x) = \sum_{i,j=1}^{d} \left( a_{ij}(x)u_{x_i}(x) \right)_{x_j} - c(x)u(x), \quad x \in D,
	\]
	where $c,a_{ij}\in C^{1}(\bar{D})$ satisfy $c\ge 0$ and $a_{ij}=a_{ji}$ for $1\le i,j\le d$. Moreover there exists a constant $a_{0}$, such that
		\[a_{0} \sum_{i=1}^{d} \xi_{i}^{2} \leq \sum_{i, j=1}^{d} a_{ij}(x) \xi_{i} \xi_{j}, \quad x \in \bar{D}, \xi \in \mathbb{R}^{d}.\]
	The real-valued functions $R(t)$, $f(x)$, and $g(x)$ represent the information of the source item. The $u_0(x)$ denotes the initial condition of the model. The term $g(x)dw$ represents spatially correlated random noise that influences the source term $f(x)$, where $g(x)$ quantifies the intensity of randomness.
	\par
	The observable data information in this paper is given by
	\begin{align*}  
		u^{(i)}(x, T) &= h^{(i)}(x),i=1, \cdots, N^{2},
		\tag{1.2}
		\label{eq:1.2} 
	\end{align*}
	where $u^{(i)}(x,T)$ represents the value of $u(x,T)$ at terminal time $T$ for the $i$-th observation (these $N^{2}$ observations are independent each other). Therefore, the stochastic inverse problem of the source term $f(x)$ is formulated based on the observed data $h^{(i)}(x)$. There is a distinction in using observed data: Bayesian inversion requires $N^{2}$ observations at $M$ discrete points in domain $D$. Comparatively, stochastic inverse problems for well-posedness analysis require $N^{2}$ observations continuously covering the entire domain $D$. The difference lies only in whether observations are discrete or continuous.
	\par
	{\bf The specific works we have completed include:}
	\par
    A new weighting formula is proposed based on the characteristics of the observed data and the structure of the Bayesian MAP optimization objective (specifically, the calculation formula of the covariance matrix). We prove that the formula is theoretically sound and converges to the standard covariance formula for independent and identically distributed (i.i.d.) observed data as the number of iterations increases. Combining the Bayesian inversion method with the new weighting formula stably solves stochastic inverse problems, thereby addressing Q1. Building on the solution to Q1, we establish an optimization condition for UQ by integrating variational inference with the new weighting formula. This condition enables efficient quantification of the uncertainty of the solution, thereby resolving Q2.
    Furthermore, we establish uniqueness theorems for the source term $f(x)$ by leveraging the statistical information of the observed data and analyzing the ill-posedness of the inverse problem. We derive error estimates for the difference between the exact and regularized solutions using regularization techniques. We further investigate the error relationship under varying amounts of observed data. Based on this error analysis, we provide a theoretical explanation for the uncertainty and compare it with the UQ obtained from numerical calculations in Q2, thereby addressing Q3. 
    \par
    Numerically, we propose an effective two-stage optimization method to address stochastic inverse problems within the Bayesian inversion framework. In the first stage, we employ the new weighting formula with the conjugate gradient method (CGM) to derive the optimal estimation solution. In the second stage, we utilize a random sampling method to refine point estimates and incorporate the uncertainty optimization condition to quantify the uncertainty of the solution. Finally, some numerical results, including optimal error estimates and UQ, are presented to demonstrate the effectiveness of the proposed method.
	\par
	The structure of this paper is as follows: Section 2 introduces the concepts and properties of solutions for stochastic inverse problems, the Bayesian inversion method, and related conclusions. Section 3 analyzes the ill-posedness and uniqueness theorems for Eq. (\ref{eq:1.1}), discusses prior estimates of the regularized solution, and outlines the construction of the adjoint gradient. Section 4 evaluates the rationale behind the new weighted formula and derives the UQ condition using Bayesian variational inference and the new weighted formula. Section 5 presents the detailed processes of the two-stage optimization method. Section 6 verifies the method's effectiveness through numerical examples. Finally, the conclusion is exhibited in Section 7.
	\section{Preliminaries}\label{sec:model}
	In this section, we review mild solutions of SPDEs and their properties, introduce the Bayesian inversion framework, and emphasize that the primary challenge in solving stochastic inverse problems within this framework is the ill-posedness of the covariance matrix of the observed data. Unlike deterministic inverse problems assuming i.i.d. observations, the ill-posedness in SPDE solutions inherently arises from their statistical characteristics. Finally, we briefly introduce variational inference as an alternative approach for the UQ.
	\subsection{Mild solution and its properties}
	We disregard the spatial variable and reformulate Eq. (\ref{eq:1.1}) into its standard evolutionary from
	\begin{equation}
		\left\{\begin{aligned}
			du&=A u d t+R f d t+g d w, & & t \in(0, T], \\
			u(0) & =u_{0}, & & t=0,
		\end{aligned}\right.
		\tag{2.1}
		\label{eq:2.1} 
	\end{equation}
	where the Brownian motion $\{w\}_{t \geq 0}$ is defined on complete filtered probability space $(\Omega, \mathcal{F},\left\{\mathcal{F}_{t}\right\}_{t \geq 0}, \mathbb{P})$.
	\par
	Since $A$ is a symmetric differential operator satisfying the uniform ellipticity condition, we assume that the eigenvalue system of the operator $-A$ with homogeneous Dirichlet boundary conditions is $\left\{\lambda_{n}, \varphi _{n}\right\}_{n=1}^{\infty}$. Thus, the non-negative and increasing sequence $\{\lambda_n\}_{n=1}^{\infty}$ satisfies: $0<\lambda_{1} \leq \cdots \lambda_{n} \leq \cdots$, and $\lambda_{n} \rightarrow \infty$ as $n \rightarrow \infty$. The eigenfunctions $\left\{\varphi_{n}(x)\right\}_{n=1}^{\infty}$ of the operator $-A$ form a complete orthonormal basis in $L^{2}(D)$, satisfying the relation $A \varphi_{n}=-\lambda_{n} \varphi_{n}, n=1, 2, \cdots$. Therefore, any function $v(x,t)$ in the $L^{2}(D)$ space can be decomposed as follows,
	\[  
	v(x, t) = \sum_{n=1}^{\infty} v_{n}(t) \varphi_{n}(x)  
	\]
	where $v_{n}(t)=(v(x,t),\varphi_{n}(x))_{L^{2}(D)}$.
	\par
	We expand $u(x,t)$ of Eq. (\ref{eq:2.1}) in the $L^{2}(D)$ space, and thus $u_n(t)$ satisfies the stochastic ordinary differential equation,
	\begin{equation}
		\left\{\begin{aligned}
			du_{n}(t)&=-\lambda_{n} u_{n}(t) d t+R f_{n} d t+g_{n}dw, && t\in(0,T], \\
			u_{n}(0)&=u_{0,n},& & t=0,
		\end{aligned}\right.
		\tag{2.2} 
		\label{eq:2.2}
	\end{equation}
	where $f_{n}=\left(f(x),\varphi_{n}(x)\right)_{L^{2}(D)}, g_{n}=\left(g(x),\varphi_{n}(x)\right)_{L^{2}(D)}$ and $u_{n}(0)=\left(u_{0},\varphi_{n}(x)\right)_{L^{2}(D)}$. Eq. (\ref{eq:2.2}) is solved as follows
	\begin{center}
		$u_{n}(t)=e^{-\lambda_{n}t} u_{0, n}+\int_{0}^{t} e^{-\lambda_{n}(t-s)} R(s) f_{n} d s+\int_{0}^{t} e^{-\lambda_{n}(t-s)} g_{n} d w $.
	\end{center}
	Therefore,
	\begin{align*}
		u(x, t) =\sum_{n=1}^{\infty}\left(e^{-\lambda_{n} t} u_{0, n}+\int_{0}^{t} e^{-\lambda_{n}(t-s)} R(s) f_{n} d s+\int_{0}^{t} e^{-\lambda_{n}(t-s)} g_{n} d w\right) \varphi_{n}(x),
		\tag{2.3}
		\label{eq:2.3}
	\end{align*}
	the definition of mild solution for Eq. (\ref{eq:2.1}) is given.
	\begin{definition}{(\cite{ref13})} 
		A $L^{2}(D)$-measurable stochastic process $u(x, t)$ is called a mild solution to Eq. (\ref{eq:2.1}) if
		\[
		u(x, t) =\sum_{n=1}^{\infty}\left(e^{-\lambda_{n} t} u_{0, n}+\int_{0}^{t} e^{-\lambda_{n}(t-s)} R(s) f_{n} d s+\int_{0}^{t} e^{-\lambda_{n}(t-s)} g_{n} d w\right) \varphi_{n}(x), \quad \mathbb{P} \text {-a.s.}
		\]
		for each $(x,t)\in (\bar{D}\times [0,T])$ and it is well-defined.
	\end{definition}
	To validate the rationality of the definition, it is necessary to examine the regularity properties of the mild solution of Eq. (\ref{eq:2.1}).
	\begin{Assumption}
		Let $u_{0}(x)\in L^{2}(D)$, $f(x)\in L^{2}(D)$, $g(x)\in L^{2}(D)$, and $R(t)\in L^{\infty }([0, T])$ with $0 < c_R \leq R(t)$ almost everywhere (a.e.) on $[0,T]$, where $c_{R}$ is a positive constant.
		\label{Ass:1}
	\end{Assumption}
	\begin{theorem}
		Under Assumption \ref{Ass:1}, let $C = \|R\|_{L^\infty([0,T])}$. Then, the solution $u(x,t)$ satisfies
		\begin{center}
			$\mathbb{E}\left[\int_{0}^{T} \|u(x,t)\|_{L^{2}(D)}^{2} dt\right] \leq 3T\|u_0\|_{L^2(D)}^2 + CT^3\|f\|_{L^2(D)}^2 + \frac{3T^2}{2}\|g\|_{L^2(D)}^2.$
		\end{center}
		\label{th:2.1}
	\end{theorem}
	\begin{proof}
		See Appendix \ref{A1-section}.
	\end{proof}	
	Under the guarantee of Theorem \ref{th:2.1}, the mild solution is well-defined. We further analyze the properties of the solution. From the superposition principle of solutions, we have
	\begin{align*}
		u(x, t) & =\sum_{n=1}^{\infty} u_{n}(t) \varphi_{n}(x) \\
		& =\sum_{n=1}^{\infty}\left(e^{-\lambda_{n} t} u_{0, n}+\int_{0}^{t} e^{-\lambda_{n}(t-s)} R(s) f_{n} d s+\int_{0}^{t} e^{-\lambda_{n}(t-s)} g_{n} d w\right) \varphi_{n}(x) \\
		& =\underbrace{\sum_{n=1}^{\infty}\left(\int_{0}^{t} e^{-\lambda_{n}(t-s)} R(s) f_{n} d s\right) \varphi_{n}(x)}_{u^{[1]}(x, t)}+\underbrace{\sum_{n=1}^{\infty}\left(e^{-\lambda_{n} t} u_{0, n}+\int_{0}^{t} e^{-\lambda_{n}(t-s)} g_{n} d w\right) \varphi_{n}(x)}_{u^{[2]}(x, t)}.
	\end{align*}
	The $u(x,t)$ is decomposed into $u (x, t) = u ^ {[1]} (x, t) + u ^ {[2]} (x, t) $. $u^{[1]} (x, t) $ and $u^{[2]}(x, t) $ are given by
	\begin{equation}
		\left\{\begin{aligned}
			d u^{[1]} & =A u^{[1]} d t+f R d t, & & t \in(0, T], \\
			u^{[1]}(0) & =0, & & t=0,
		\end{aligned}\right.
		\tag{2.4} 
		\label{eq:2.4}
	\end{equation}
	\begin{equation}
		\left\{\begin{aligned}
			d u^{[2]} & =A u^{[2]} d t+g d w, & & t \in(0, T], \\
			u^{[2]}(0) & =u_{0}, & & t=0,
		\end{aligned}\right.
		\tag{2.5}
		\label{eq:2.5} 
	\end{equation}
	and satisfy
	\begin{align*}
		u=u^{[1]}+u^{[2]} \sim \mathcal{N}\left(h^{*}, D^*\right),
		\tag{2.6} 
		\label{eq:2.6}
	\end{align*}
	where $h^{*}= \mathbb{E}\left[u\right]$, $ D^*=\operatorname{var}[u]$. Based on Itô's formula and (\ref{eq:2.3}), we have
	\begin{align*} 
		\mathbb{E}\left[u\right]&=\sum_{n=1}^{\infty}\left(e^{-\lambda_{n} t} u_{0, n}+\int_{0}^{t} e^{-\lambda_{n}(t-s)} R(s) f_{n} d s\right)\varphi_{n}(x),\\
		\operatorname{var}\left[u\right]&=\mathbb{E}\left(u -\mathbb{E}\left[u\right]\right)^{2}=\mathbb{E}\left(\sum_{n=1}^{\infty} \int_{0}^{t} e^{-\lambda_{n}(t-s)} g_{n} d w \varphi_{n}(x)\right)^{2} \\
		&=\int_{0}^{t}\left(\sum_{n=1}^{\infty} e^{-\lambda_{n}(t-s)} g_{n} \varphi_{n}(x)\right)^{2} ds.
		\tag{2.7}
		\label{eq:2.7}
	\end{align*}
    In order to verify $\operatorname{var}[u]$ is well-defined, we integrate $\operatorname{var}[u]$ over $D$,
    \begin{align*}
    	\int_{D} \operatorname{var}[u] d x & =\int_{D}\int_{0}^{t}\left(\sum_{n=1}^{\infty} e^{-\lambda_{n}(t-s)} g_{n} \varphi_{n}(x)\right)^{2} d s d x \\
    	& =\sum_{n=1}^{\infty} \int_{0}^{t} e^{-2 \lambda_{n}(t-s)} d s\left(g_{n}\right)^{2} \leq \frac{\|g\|_{L^{2}(D)}^{2}}{2 \lambda_{1}} < \infty,
    \end{align*}
    where $\lambda_{1}$ is eigenvalue, thereby $\operatorname{var}[u] < \infty$ a.e. in $x \in D$. 
    Additionally, by Cauchy-Schwarz inequality, for any $x_1,x_2 \in D$, $\Sigma(u_1,u_2) = \operatorname{Cov}\left[u\left(x_{1}, t\right), u\left(x_{2}, t\right)\right] \leq \operatorname{var}\left[u\left(x_{1}, t\right)\right] \operatorname{var}\left[u\left(x_{2}, t\right)\right]<\infty$, and satisfy
    \begin{align*}
    	\operatorname{Cov}\left[u\left(x_{1}, t\right), u\left(x_{2}, t\right)\right] & =E\left[\left(u\left(x_{1}, t\right)-E\left[u\left(x_{1}, t\right)\right]\right)\left(u\left(x_{2}, t\right)-E\left[u\left(x_{2}, t\right)\right]\right)\right] \\
    	& =\sum_{m, n=1}^{\infty} \frac{1-e^{-\left(\lambda_{n}+\lambda_{m}\right) t}}{\lambda_{n}+\lambda_{m}} g_{n} g_{m} \varphi_{n}\left(x_{1}\right) \varphi_{m}\left(x_{2}\right) d s <\infty.
    	\tag{2.8}
    	\label{eq:2.8}
    \end{align*}
    \par
    From (\ref{eq:2.7}) and (\ref{eq:2.8}), we know that the statistical characteristics of the mild solution $u(x,t)$ to Eq. (\ref{eq:2.1}) vary with $(x, t)$ (\cite{ref28}). In particular, var$[u(x,T)]$ tends to zero as $x$ approaches the boundary $\partial D$ (where $\varphi_n(x) = 0$ on $\partial D$). Consequently, the covariance matrix $\Sigma$ constructed from the observed data in (\ref{eq:1.2}) may become ill-conditioned or even singular, especially near $\partial D$. Notably, ill-conditioning in $\Sigma$ leads to numerical instability in Bayesian MAP estimation (e.g., via(\ref{eq:2.12})), causing significant inference errors and potential computational failure.
    \par
    On the other hand, substituting the observed data (\ref{eq:1.2}) into (\ref{eq:2.6}), we obtain
    \begin{align*}  
    	h^{(i)}(x) & \stackrel{d}{=} \mathbb{E}[u(x, T)] + \sqrt{D^*(x,T)} \xi_{i}, \quad i=1, \cdots, N^{2}, \\
    	h^{N}(x) = &\frac{1}{N^2}\sum_{i=1}^{N^2} h^{(i)}(x) \stackrel{d}{=} \mathbb{E}[u(x, T)] + \frac{\sqrt{\operatorname{var}[u(x,T)]}}{N}\xi_{0},
    \end{align*}
    where $h^{(i)}(x)$ and $h^{N}(x)$ represent the $i$-th observation and the sample mean of the $N^2$ obsered data, respectively. The $a \stackrel{d}{=} b$ indicates that $a$ and $b$ are equal in the sense of distribution (\cite{ref8}). Here, $\xi_{i}$ and $\xi_{0}$ are independent standard Gaussian random variables. Further, define
	\begin{align*}
		h^{\delta,N} \stackrel{d}{=} h^{N}(x)-\mathbb{E}[u^{[2]}(x,T)]+\frac{\sqrt{\operatorname{var}[u(x, T)]}}{N}\xi_{0},
		\tag{2.9}
		\label{eq:2.9}
	\end{align*}
    \[\mathbb{E}[h^{\delta,N}] = \mathbb{E}[u(x,T)]-\mathbb{E}[u^{[2]}(x,T)],\]
	where $h^{\delta, N}$ is derived by subtracting the expected data $\mathbb{E}(u^{[2]}(x,T))$ from the observed data (\ref{eq:1.2}). Combining Eq. (\ref{eq:2.4}), Eq. (\ref{eq:2.5}) and (\ref{eq:2.9}), we obtain generally framework:   
	\begin{align*}
		h^{\delta,N}=F(f)+\frac{1}{N} \xi,
		\tag{2.10}
		\label{eq:2.10}
	\end{align*}
	where $h^{\delta, N}$ is necessary observed data for inverting the source term $f$ in (\ref{eq:2.9})), $F(f) = u^{[1]}(x, T)$ is the solution operator of Eq. (\ref{eq:2.4}), and $\xi \sim \mathcal{N}\left(0, \operatorname{var}\left(u^{[2]}(x, T)\right)\right)$ represents the uncertainty from observed data $h^{\delta, N}$ and reflected through Eq. (\ref{eq:2.5}). Consequently, solving the stochastic inverse problem of Eq. (\ref{eq:2.1}) is equivalent to solving the problem (\ref{eq:2.10}). This formula is consistent with the Bayesian inversion framework, and we will describe it in detail in the following subsection.
	\subsection{Bayesian inversion method}
	Since the inversion target $f(x) \in L^2(D)$, conventional finite-dimensional Bayesian methods cannot characterize its properties. We adopt the framework of infinite-dimensional Bayesian inversion to study the problem (\ref{eq:2.10}); more details can be found in \cite{ref4,ref22}.
	\subsubsection{Infinite-dimensional Bayesian inversion framework with $N^{2}$ observed data}
	Let $X$ and $Y$ be two separable Hilbert spaces, in which the inner product and norm are defined as $(\cdot, \cdot)$ and $\left\| \cdot \right\|$, respectively. The weighted inner product and norm are defined by 	$\langle\cdot, \cdot\rangle_{A}=\left\langle A^{-\frac{1}{2}}, A^{-\frac{1}{2}} \cdot\right\rangle$ and $\|\cdot\|_{A}=\left\|A^{-\frac{1}{2}} \cdot\right\|$. Let $F: X \rightarrow Y$ denote a bounded linear operator. Following the ideas presented in \cite{ref4}, we wish to solve the inverse problem of finding $f$ from observed data $h^{\delta, N}$ where
	\begin{align*}
		h^{\delta,N}=F(f)+\frac{1}{N} \xi,
		\tag{2.11}
		\label{eq:2.11}
	\end{align*} 
	and $\xi \in Y$ represent the random error, independent of $f$. For simplicity, we take $Y = \mathbb{R}^M$, which does not conflict with (\ref{eq:2.10}) (relevant explanations can be found in observed data (\ref{eq:1.2})). The infinite-dimensional Bayesian inversion framework under $N^{2}$ observed data as following:
	\begin{lemma} {(\cite{ref22})}
		Suppose that $F: X \rightarrow \mathbb{R}^{M}$ is a bounded linear operator, and the prior measure $\mu_{0}$ is a Gaussian measure satisfying $\mu_{0}(X) = 1$, then the solution $f| h^{\delta, N}$ of (\ref{eq:2.11}) exists and when the posterior measure $\mu^{h^{\delta, N}}$ is absolutely continuous with respect to the prior measure $\mu_{0}$, the posterior distribution can be given by the Radon-Nikodym derivative
		\begin{center}
			$\frac{d\mu^{h^{\delta,N}}}{d\mu_{0}}(f)=\frac{1}{Z(h^{\delta,N})} \exp \left(-\frac{N^{2}}{2}\|h^{\delta,N}-F(f)\|_{\Sigma}^{2}\right)$,
		\end{center}
		where $Z(h^{\delta,N})=\int_{X} \exp \left(-\frac{N^2}{2}\|h^{\delta,N}-F(f)\|_{\Sigma}^{2}\right) \mu_{0}(f)>0$.
		\label{lemma:2.1}
	\end{lemma}
	According to the Lemma \ref{lemma:2.1}, we obtain
	\begin{align*}
		\frac{d \mu^{h^{\delta, N}}}{d\mu_{0}}(f) & =\frac{1}{Z\left(h^{\delta, N}\right)} \exp \left(-\frac{N^{2}}{2}\left\|h^{\delta, N}-F(f)\right\|_{\Sigma}^{2}\right) \\
		& \propto \exp \left(-\frac{N^{2}}{2}\left\|h^{\delta, N}-F(f)\right\|_{\Sigma}^{2}\right).
	\end{align*}
	\par
	The Bayesian MAP estimation of $f$ is equivalent to solving the following minimization problem:
	\begin{align*}
		min \quad J=\frac{1}{2}\left\|h^{\delta, N}-F(f)\right\|_{\Sigma}^{2}+\frac{1}{2 N^{2}}\left\|f-f^{0}\right\|_{C_{0}}^{2},
		\tag{2.12}
		\label{eq:2.12}
	\end{align*}
	where the prior distribution is $\mu_{0} \sim \mathcal{N}(f^{0}, C_{0})$. For i.i.d. observed data, the covariance matrix $\Sigma$ reduces to a scalar matrix $\Sigma = \sigma^{2} I$, with $I \in R^{M \times M}$. The minimization objective (\ref{eq:2.12}) then becomes
	\begin{align*}
		min \quad \frac{1}{2}\left\|h^{\delta, N}-F(f)\right\|_{2}^{2}+\frac{\sigma^{2}}{2 N^{2}}\left\|f-f^{0}\right\|_{C_{0}}^{2}.
		\tag{2.13}
		\label{eq:2.13}
	\end{align*}
    Evidently, for i.i.d. observed data, the Bayesian MAP estimation corresponds to minimizing (\ref{eq:2.13}), which is analogous to Tikhonov regularization for inverse problems (though regularization parameters may differ (\cite{ref6})), and calculate (\ref{eq:2.11}) is stable. However, for non-i.i.d. observed data (e.g., as in (\ref{eq:2.8})) may become ill-posed due to the ill-conditioned covariance matrix $\Sigma$, thereby making the (\ref{eq:2.12}) is calculation unstable. Thus, developing effective strategies to regularize $\Sigma$ is crucial.
    \par
    Furthermore, (\ref{eq:2.12}) and (\ref{eq:2.13}) inherently lack the UQ. To overcome this limitation, Monte Carlo sampling methods are commonly employed to approximate the complete posterior distribution. However, such methods are computationally expensive. In contrast, variational inference offers a more efficient alternative by approximating the posterior distribution through optimization techniques, thereby avoiding the challenges of direct posterior computation (\cite{ref16}). Variational inference has become a critical tool for approximating Bayesian posterior distributions. It has been widely applied in computer science and statistics (\cite{ref21, ref23}). Further details on this approach are discussed in Section 4.
	\section{The stochastic inverse problem}
    In the previous section, we outlined the key challenges in studying stochastic inverse problems using the Bayesian inversion methods as the Bayesian approach is well-established (\cite{ref4,ref22}). Therefore, we focus instead on the proof of the uniqueness theorem for the inversion target $f$ based on the statistical properties of the observed data. Additionally, we examine the ill-posedness of this inverse problem and propose a Tikhonov regularization to address it. We derive error estimates between the regularized resolution and the exact solution of $f$. Finally, the adjoint gradient of the Tikhonov regularization functional was derived, providing a theoretical basis for the numerical inversion of $f$.
	\subsection{Uniqueness}
	To reconstruct the source term $f(x)$. We combine (\ref{eq:2.3}) and (\ref{eq:2.9}) at $t=T$,
    \begin{align*}
        \mathbb{E}[u^{[1]}(x, T)] & = \mathbb{E}[u(x, T)] -\mathbb{E}[u^{[2]}(x,T)]\\ 
        & = \sum_{n=1}^{\infty}\left(\int_{0}^{T} e^{-\lambda_{n}(T-s)} R(s) f_{n} d s\right) \varphi_{n}(x)\\
    	& = \mathbb{E}[u(x, T)] -\sum_{n=1}^{\infty}\left( e^{-\lambda_{n} T} u_{0, n}\right) \varphi_{n}(x).
    	\tag{3.1}
    	\label{eq:3.1}
    \end{align*}
    Define $\mathbb{E}[u_{n}(T)] = \left(\mathbb{E}[u(x,T)],\varphi_{n}(x)\right)_{L^{2}(D)} = e^{-\lambda_{n}T} u_{0, n}+\int_{0}^{T} e^{-\lambda_{n}(T-s)} R(s) f_{n} d s$. The source term $f(x)$ can then be uniquely determined as:
    \begin{align*}
    	f(x)=\sum_{n=1}^{\infty} \frac{\mathbb{E}[u_{n}(T)] - e^{-\lambda_{n}T} u_{0, n}}{\int_{0}^{T} e^{-\lambda_{n}(T-s)} R(s) d s} \varphi_{n}(x).
    	\tag{3.2}
    	\label{eq:3.2}
    \end{align*}
    \par
    The uniqueness of $f(x)$ follows directly from (\ref{eq:3.2}) when the expectations of the observed data are determined.
    \begin{theorem}
    	Under Assumptions \ref{Ass:1}, if solutions $u_1$ and $u_2$ of Eq. (\ref{eq:2.1}) corresponding to the source term $f_1$ and $f_2$ satisfies
    	\begin{align*}
    		\mathbb{E}[u_1(x,T)] = \mathbb{E}[u_2(x,T)] \quad \forall x \in D,
    	\end{align*}
    	then $ f_1(x) = f_2(x) $ a.e. in $ D $.
    	\label{th:3.1}
    \end{theorem}
    \begin{proof}
    	Given $ R(t) \geq c_R > 0 $ a.e. on $[0,T]$, for each $ n \geq 1 $,
    	\begin{align*}
    		\int_0^T e^{-\lambda_n(T-s)} R(s) ds \geq c_R \int_0^T e^{-\lambda_n(T-s)} ds = \frac{c_R}{\lambda_n} \left(1 - e^{-\lambda_n T}\right) > 0.
    	\end{align*}
     Expand solutions $ u_1 $ and $ u_2 $ as
    	\begin{align*}
    		u_i(x,T) = \sum_{n=1}^\infty u_{i,n}(T) \varphi_n(x) \quad (i=1,2),
    	\end{align*}
    	the expectation of their coefficients satisfies
    	\begin{align*}
    		\mathbb{E}[u_{1,n}(T) - u_{2,n}(T)] = (f_{1,n} - f_{2,n}) \int_0^T e^{-\lambda_n(T-s)} R(s) ds.
    	\end{align*}
    	If $ \mathbb{E}[u_1(x,T)] = \mathbb{E}[u_2(x,T)] $, then $ f_{1,n} = f_{2,n} $ for all $ n $. By the completeness of $ \{\varphi_n\} $ in $ L^2(D) $, we conclude $ f_1 = f_2 $.
    \end{proof}
	\subsection{Regularization}
	In this subsection, we demonstrate that the inverse problem is unstable to recover $f$. Given $ R(t)\in L^{\infty }([0, T])$ with $0<R(t)\le C$ a.e. on $[0,T]$, where $C = \|R\|_{L^\infty([0,T])}$. A simple calculation yields
	\[
	f(x)=\sum_{n=1}^{\infty} \frac{\mathbb{E}[u_{n}(T)] - e^{-\lambda_{n}T} u_{0, n}}{\int_{0}^{T} e^{-\lambda_{n}(T-s)} R(s) d s} \varphi_{n}(x)\ge \sum_{n=1}^{\infty} \frac{\lambda_{n}\left ( \mathbb{E}[u_{n}(T)] - e^{-\lambda_{n}T} u_{0, n} \right ) }{C\left(1-e^{-\lambda_{n} T}\right)} \varphi_{n}(x).
	\]
	\par
	This inequality reveals the ill-posedness of directly inverting $f$ via (\ref{eq:3.2}). Specifically, as $n \to \infty$, $\lambda_n \to \infty$ causes even minor errors in the observed data to be amplified. Therefore, we utilize the regularization method to overcome this difficulty, thereby establishing the error estimation of exact $f$ and its regularization resolution for the given $N^2$ observed data (\ref{eq:1.2}).
	\par
	Firstly, define the solution operator of Eq. (\ref{eq:2.4}) as $F: L^2(D) \to L^2(D)$ given by
	\begin{align*}
		F (f)(x)=\sum_{n=1}^{\infty}\left(\int_{0}^{T} e^{-\lambda_{n}(T-s)} R(s) f_{n} d s\right) \varphi_{n}(x),
		\tag{3.3}
		\label{eq:3.3}
	\end{align*}
	where $f_{n}=\left(f(x),\varphi_{n}(x)\right)_{L^{2}(D)}$. It is not difficult to verify, $F$ is a linear, symmetric self-adjoint operator,  and $\left \| F(f) \right \| \le CT\left \| f \right \| $, where $C = \|R\|_{L^\infty([0,T])}$. Secondly, to align with the optimization framework described in $(2.12)$, we adopt the Tikhonov regularization method to construct the minimization functional:
	\begin{align*}
		J_{\gamma}[f]=\frac{1}{2}\left\|F(f)-h^{\delta,N}\right\|_{L^{2}(D)}^{2}+\frac{\gamma}{2}\|f-f^{0}\|_{L^{2}(D)}^{2},
		\tag{3.4}
		\label{eq:3.4}
	\end{align*}
    where, $ h^{\delta,N}$ is the sample mean of observed data in (\ref{eq:2.9}), $ \gamma > 0 $ is a regularization parameter and $ f^{0}$ is prior information of $f$. $J_{\gamma}$ is a random functional due to $h^{\delta,N}$. When $h^{\delta, N}$ is fixed, we solve the deterministic optimization problem. Additionally, compared to the Bayesian MAP estimation (\ref{eq:2.12}), this continuous formulation (\ref{eq:3.4}) facilitates the derivation of the regularized solution and adjoint gradients.
	\begin{remark}
		Optimize perspective: The uniqueness of the minimizer of (\ref{eq:3.4}) follows from the linearity and boundedness of $F$, for fixed each $h^{\delta, N}$. Bayesian perspective: For Gaussian priors and linear bounded $F$, the posterior distribution of $f$ is uniquely Gaussian (\cite{ref4,ref24}). This provides a theoretical basis for selecting distribution clusters in variational inference (Section 4).
	\end{remark}
	\par
	Thirdly, For each $h^{\delta,N}$, let $f_{\gamma}^{\delta}(x)$ be corresponding the minimizer of (\ref{eq:3.4}), which satisfies the following normal equation:
	\begin{align*}
		F^{*} F f_{\gamma }^{\delta}+\gamma f_{\gamma }^{\delta}=F^{*} h^{\delta,N} + \gamma f^{0},
		\tag{3.5}
		\label{eq:3.5}
	\end{align*}
	where (\ref{eq:3.5}) holds almost surely (a.s.) for each $h^{\delta,N}$. Here $F^{*} = F$ represent the conjugate operator of $F$ ( self-adjoint). Substituting (\ref{eq:3.2}) and (\ref{eq:3.3}) into Eq. (\ref{eq:3.5}) yields
	\begin{align*}
		f_{\gamma}^{\delta}=\frac{F^{*} h^{\delta,N}+\gamma f^{0}}{F^{*} F+\gamma}=\sum_{n=1}^{\infty}\left(\frac{\int_{0}^{T} e^{-\lambda_{n}(T-s)} R(s) d s h_{n}^{\delta} + \gamma f^{0}_{n}}{\left(\int_{0}^{T} e^{-\lambda_{n}(T-s)} R(s) d s\right)^{2}+\gamma}\right) \varphi_{n}(x),
		\tag{3.6}
		\label{eq:3.6}
	\end{align*}
	where $h_{n}^{\delta} = \left(h^{\delta,N}(x), \varphi_{n}(x)\right)_{L^{2}(D)} = h^{N}_{n} - e^{-\lambda_{n}T} u_{0, n}$ (From (\ref{eq:2.9}), random coefficient) and $f_{n}^{0}=\left(f^{0}(x), \varphi_{n}(x)\right)_{L^{2}(D)}$. Further, we consider in (\ref{eq:3.4}), $h^{\delta,N}$ is a deterministic function, that is, $h=\mathbb{E}[u^{[1]}(x, T)]$. Similarly, the solution is denoted by
	\begin{align*}
		f_{\gamma}=\frac{F^{*}h+\gamma f^{0}}{F^{*} F+\gamma}=\sum_{n=1}^{\infty}\left(\frac{\int_{0}^{T} e^{-\lambda_{n}(T-s)} R(s) d s h_{n} +  \gamma f^{0}_{n}}{\left(\int_{0}^{T} e^{-\lambda_{n}(T-s)} R(s) d s\right)^{2}+\gamma}\right) \varphi_{n}(x),
		\tag{3.7}
		\label{eq:3.7}
	\end{align*}
	where  $h_{n} = \left(h(x), \varphi_{n}(x)\right)_{L^{2}(D)} = E\left[u_{n}(T)\right] - e^{-\lambda_{n}T} u_{0, n}$ (deterministic coefficient). 
	\par
	Finally, we need from (\ref{eq:2.8}) to estimate the $L^2$-error between the sample mean of $N^2$ observations and the expected value of the solution of Eq. (\ref{eq:2.1}) at $t=T$,
	\begin{align*}
		\mathbb{E}\left[\left\|h^{N, \delta}-h\right\|_{L^{2}(D)}^{2}\right] &= \mathbb{E}\left[ \left\| h^{N}(x) - \mathbb{E}[u(x, T)]\ \right \|_{L^{2}(D)}^{2} \right] = \mathbb{E}\left[ \left\| \frac{\sqrt{\operatorname{var}[u(x, T)]}}{N}\xi_{0} \right\|_{L^{2}(D)}^{2} \right] \\
		&= \mathbb{E}\left[ \int_{D} \frac{\operatorname{var}[u(x, T)]}{N^{2}}\xi_{0}^2  dx \right] = \int_{D} \frac{\operatorname{var}[u(x, T)]}{N^{2}} \mathbb{E}[\xi_{0}^2]  dx \text{ (by Fubini's theorem)} \\
		&= \frac{1}{N^{2}} \int_{D} \operatorname{var}[u(x, T)]  dx \leq \frac{1}{N^{2}} \cdot \frac{\|g\|_{L^{2}(D)}^{2}}{2\lambda_{1}} \quad \text{(by the variance bound)} \\
		&= \frac{C_\lambda}{N^{2}} \|g\|_{L^{2}(D)}^{2}, 
		\tag{3.8} 
		\label{eq:3.8}
	\end{align*}
	where $\xi_{0}$ be standard Gaussian random variable, $C_\lambda= \frac{1}{2\lambda_{1}}$ and $\lambda_{1}$ is the eigenvalue of the operator associated with Eq. (\ref{eq:2.1}). Thereby, the error estimates between the exact solution and the regularized solution are derived.
	\begin{Assumption}
	    let \(h(x), h^{\delta,N}(x) \in L^{2}(D)\) and $(f(x)-f^{0}(x)) \in H^2(D)$ , where $h(x)=\mathbb{E}[u^{[1]}(x, T)]$ represent expectation of the solution for Eq. (\ref{eq:2.4}) (deterministic), $h^{\delta,N}(x)$ be the sample mean from $N^2$ random observations (stochastic) and \(f^{0}\) be the prior of $f$.
		\label{Ass:2}
	\end{Assumption}
  
	\begin{theorem}
		Under Assumptions \ref{Ass:1} and \ref{Ass:2}. Define \(\delta = \frac{\sqrt{C_\lambda}}{N} \|g\|_{L^{2}(D)}\) with \(C_\lambda = \frac{1}{2\lambda_{1}}\). Assume $f - f^0 \in H^2(D)$. Then, there exist constants \(C, C_{1} > 0\) such that if $\gamma =C_1\delta ^{\frac{2}{3}}$, the exact solution $f(x)$ in Eq. (\ref{eq:2.1}) and the minimizer $f_{\gamma}^{\delta}(x)$ of (\ref{eq:3.4}) satisfies
		\[
		\mathbb{E} \left[ \left\| f(x) - f_{\gamma}^{\delta}(x) \right\|_{L^{2}(D)}^2 \right] \leq C \delta^{\frac{4}{3} },
		\]
		where $C=\left( \frac{2K^2 C_1^2}{(C_R C_0)^4} + \frac{1}{2C_1} \right)$ with, $K=\left \| f(x)-f^{0}(x) \right \|_{H^2(D)} $, $C_0 = 1 - e^{-\lambda_1 T}$, $0 < C_R \leq R(t)$ on $[0,T]$.
		\label{th:3.2}
	\end{theorem}

	\begin{proof}
		By the Minkowski inequality,
		\begin{center}
			$\mathbb{E}\left[\|f(x) - f_{\gamma}^{\delta}(x)\|_{L^2}^2\right] \leq 2 \mathbb{E}\left[\|f(x) - f_{\gamma}(x)\|_{L^2}^2\right] + 2 \mathbb{E}\left[\|f_{\gamma}(x) - f_{\gamma}^{\delta}(x)\|_{L^2}^2\right]$.
		\end{center}
		 Step 1: Estimate $\mathbb{E}\left[\left\|f(x)-f_{\gamma}(x)\right\|_{L^{2}(D)}^{2}\right]$. 
		 \par
		 Since $f(x)$ and $f_\gamma(x)$ are deterministic, using $h_{n} = f_{n} \int_{0}^{T} e^{-\lambda_{n}(T-s)} R(s) d s$, (\ref{eq:3.2}) and (\ref{eq:3.7}), we obtain
		 \begin{align*}
		 	\mathbb{E}\left[\|f(x) - f_\gamma(x)\|_{L^2}^2\right] &= \|f(x) - f_\gamma(x)\|_{L^2}^2 \\
		 	&= \sum_{n=1}^{\infty} \left( \frac{\gamma\left(h_{n}-f_{n}^{0} \int_{0}^{T} e^{-\lambda_{n}(T-s)} R(s) d s\right)}{\int_{0}^{T} e^{-\lambda_{n}(T-s)} R(s) d s\left[\left(\int_{0}^{T} e^{-\lambda_{n}(T-s)} R(s) d s\right)^{2}+\gamma\right]} \right)^2 \\
		 	&= \sum_{n=1}^{\infty} \frac{\gamma^{2}}{\left[\left(\int_{0}^{T} e^{-\lambda_{n}(T-s)} R(s) d s\right)^{2}+\gamma\right]^{2}}\left(f_{n}-f_{n}^{0}\right)^{2}.
		 \end{align*}
		By $0 < c_R \leq R(t)$ on $[0,T]$, and the calculation of the integral,
		\[
		\int_{0}^{T} e^{-\lambda_{n}(T-s)} R(s) d s \geq C_{R} \int_{0}^{T} e^{-\lambda_{n}(T-s)} d s=\frac{C_{R}\left(1-e^{-\lambda_{n} T}\right)}{\lambda_{n}}\ge \frac{C_{R}\left(1-e^{-\lambda_{1} T}\right)}{\lambda_{n}}.
		\]
		Let $C_0 =1-e^{-\lambda_{1}T}$. Then,
		\begin{align*}
			\frac{\gamma^{2}}{\left(\left(\int_{0}^{T} e^{-\lambda_{n}(T-s)} R(s) d s\right)^{2}+\gamma\right)^{2}} &\le \frac{\gamma^{2}}{\left(\int_{0}^{T} e^{-\lambda_{n}(T-s)} R(s) d s\right)^{4}} \le \frac{\lambda_{n}^{4} \gamma^{2}}{\left(C_{R} C_{0}\right)^{4}}.
		\end{align*}
	Therefore, since $f - f^{0} \in H^{2}(D)$ and \[\|f - f^0\|_{H^2(D)}^2 = \sum_{n=1}^\infty \lambda_n^4 (f_n - f_n^0)^2 = K^2 < \infty,\]
	\begin{align*}
		\|f(x) - f_\gamma(x)\|_{L^2}^2 \le \frac{ \gamma^{2}}{\left(C_{R} C_{0}\right)^{4}} \sum_{n=1}^{\infty}\lambda_{n}^{4}\left(f_{n}-f_{n}^{0}\right)^{2}\le \frac{K^2\gamma^{2}}{\left(C_{R} C_{0}\right)^{4}}
		\tag{3.9}
		\label{eq:3.9}
	\end{align*}
	Step2: Estimate \(\mathbb{E}\left[\left\|f_{\gamma}^{\delta}(x)-f_{\gamma}(x)\right\|_{L^{2}(D)}^{2}\right]\). 
	\par
	From (\ref{eq:3.6}) and (\ref{eq:3.7}), we have
		\begin{align*}
			\left\|f_{\gamma}^{\delta}(x)-f_{\gamma}(x)\right\|_{L^{2}(D)}^{2} & =\sum_{n=1}^{\infty}\left(\frac{\int_{0}^{T} e^{-\lambda_{n}(T-s)} R(s) d s h_{n}^{\delta}}{\left(\int_{0}^{T} e^{-\lambda_{n}(T-s)} R(s) d s\right)^{2}+\gamma}-\frac{\int_{0}^{T} e^{-\lambda_{n}(T-s)} R(s) d s h_{n}}{\left(\int_{0}^{T} e^{-\lambda_{n}(T-s)} R(s) d s\right)^{2}+\gamma}\right)^{2} \\
			& =\sum_{n=1}^{\infty}\left(\frac{\int_{0}^{T} e^{-\lambda_{n}(T-s)} R(s) d s}{\left(\int_{0}^{T} e^{-\lambda_{n}(T-s)} R(s) d s\right)^{2}+\gamma}\right)^{2}\left(h_{n}^{\delta}-h_{n}\right)^{2} \\
			&\leq \sum_{n=1}^{\infty}\left(\frac{\left(\int_{0}^{T} e^{-\lambda_{n}(T-s)} R(s) d s\right)^{2}}{4 \gamma\left(\int_{0}^{T} e^{-\lambda_{n}(T-s)} R(s) d s\right)^{2}}\right)\left(h_{n}^{\delta}-h_{n}\right)^{2}=\frac{1}{4 \gamma} \sum_{n=1}^{\infty}\left(h_{n}^{\delta}-h_{n}\right)^{2}.
		\end{align*}
	Taking expectation and using (\ref{eq:3.8}),
		\[
		\mathbb{E}\left[\|f_{\gamma}^{\delta} - f_{\gamma}\|_{L^2}^2\right] \leq \frac{\mathbb{E}\left[\|h^{\delta,N} - h\|_{L^2}^2\right]}{4 \gamma} \leq \frac{ \delta^2}{4 \gamma}.
		\tag{3.10}
		\label{eq:3.10}
		\]
		Step3: Combining (\ref{eq:3.9}) and (\ref{eq:3.10}),
		\[
		\mathbb{E}\left[\|f - f_{\gamma}^{\delta}\|_{L^2}^2\right] \leq 2 \mathbb{E}\left[\|f - f_{\gamma}\|_{L^2}^2\right] + 2 \mathbb{E}\left[\|f_{\gamma} - f_{\gamma}^{\delta}\|_{L^2}^2\right] \leq \frac{2K^2}{(C_R C_0)^4} \gamma^2 + \frac{ \delta^2}{2 \gamma}.
		\]
		Let $\gamma =C_1\delta ^{\frac{2}{3}}$,
		\[
		\mathbb{E}\left[\|f - f_{\gamma}^{\delta}\|_{L^2}^2\right] \leq \left( \frac{2K^2 C_1^2}{(C_R C_0)^4} + \frac{1}{2C_1} \right) \delta^{\frac{4}{3}} = C \delta^{\frac{4}{3} }.
		\]
		\par
		Thereby, we complete the proof of the Theorem \ref{th:3.2}.
	\end{proof}

	\subsection{The adjoint gradient method}
	Next, we use the gradient optimization algorithm to solve (\ref{eq:3.4}). Firstly, we substituting $F(f) = u[f]$ into (\ref{eq:3.4}), it has
	\begin{align*}
		J_{\gamma}[f]=\frac{1}{2}\left\|u[f](x,T)-h^{\delta,N}(x)\right\|_{L^{2}(D)}^{2}+\frac{\gamma}{2}\|f-f^{0}\|_{L^{2}(D)}^{2}.
		\tag{3.11}
		\label{eq:3.11}
	\end{align*}
	\par
	In order to derive the Fréchet derivative of (\ref{eq:3.11}) using the adjoint method. Let $u[f](x, t)$ be the solution of Eq. (\ref{eq:2.4}) and $\tilde{u}[f + \alpha \hat{f}](x, t)$ be the solution of the following equation:
	\begin{equation}   
		\left\{\begin{aligned} 
			d\tilde{u}[f+\alpha \hat{f}] &= A \tilde{u}[f+\alpha \hat{f}] d t + (f+\alpha \hat{f}) R(t) d t, && (x, t) \in D \times(0, T], \\
			\tilde{u}[f+\alpha \hat{f}] &= 0, && (x, t) \in \partial D \times(0, T], \\
			\tilde{u}[f+\alpha \hat{f}] &= 0, && x \in \bar{D}, t = 0,
		\end{aligned}\right.
		\tag{3.12} 
		\label{eq:3.12} 
	\end{equation}
	where $\hat{f}$ is a small perturbation such that $u[f + \alpha \hat{f}](x, t)$ is well-defined and $\alpha$ is a constant. 
	\par
	Given \[\hat{u}(x, t) = \lim_{\alpha \to 0} \frac{\tilde{u}[f + \alpha \hat{f}] - u[f]}{\alpha},\] it satisfies
	\begin{equation}
		\left\{\begin{aligned}
			d \hat{u}(x, t) & =A \hat{u}(x, t) d t+\hat{f}(x) R(t) d t, & & (x, t) \in D \times(0, T], \\
			\hat{u}(x, t) & =0, & & (x, t) \in \partial D \times(0, T], \\
			\hat{u}(x, t) & =0, & & x\in \bar{D}, t=0.
		\end{aligned}\right.
		\tag{3.13}
		\label{eq:3.13}
	\end{equation}
	Multiplying Eq. (\ref{eq:3.13}) by $P(x, t)$ and integrating over $D \times [0, T]$, we obtain
	\begin{align*}
		\int_{D} \int_{0}^{T} P d\hat{u} d x=\int_{D} \int_{0}^{T} P A\hat{u} d t d x+\int_{D} \int_{0}^{T} P R(t) \hat{f} d t d x.
		\tag{3.14}
		\label{eq:3.14}
	\end{align*}
	Using integration by parts, the left-hand side of (\ref{eq:3.14}) can be transformed
	\begin{align*}
		\int_{D} \int_{0}^{T} P d \hat{u} d x=\int_{D}\left[\left.P \hat{u}\right|_{0} ^{T}-\int_{0}^{T} \hat{u} d P\right] d x.
		\tag{3.15}
		\label{eq:3.15}
	\end{align*}
	According to the divergence theorem, the first term on the right-hand side of (\ref{eq:3.14}) is
	\begin{align*}
		\int_{D} \int_{0}^{T} P A \hat{u} d t d x =\int_{0}^{T} \left(\sum_{i, j=1}^{d}\left[P \hat{u}_{i}\right]_{\partial D}-\left[P_{j} \hat{u}\right]_{\partial D} +\int_{D} \hat{u} APd x \right) d t.
		\tag{3.16}
		\label{eq:3.16}
	\end{align*}
	Substituting (\ref{eq:3.15}) and (\ref{eq:3.16}) into (\ref{eq:3.14}), it holds
	\begin{align*}
		\int_{D}\left[\left.P \hat{u}\right|_{0} ^{T}-\int_{0}^{T} \hat{u} d P\right] d x=\int_{0}^{T} \left(\sum_{i, j=1}^{d}\left[P \hat{u}_{i}\right]_{\partial D}-\left[P_{j} \hat{u}\right]_{\partial D} +\int_{D} \hat{u} A P d x \right) d t+\int_{D} \int_{0}^{T} P R(t)\hat{f} d t dx.
		\tag{3.17}
		\label{eq:3.17}
	\end{align*}
	\par
	On the other hand, the Fréchet derivative of $J_{\gamma}[f]$ in (\ref{eq:3.11}) can be expressed as
	\begin{equation}
		\begin{aligned}
			&\left(\nabla J_{r}[f], \widehat{f}\right)_{L^{2}(D)}=\lim _{\alpha \rightarrow 0} \frac{J_{r}[f+\alpha \hat{f}]-J_{r}[f]}{\alpha} \\
			& =\lim _{\alpha \rightarrow 0} \frac{1}{2 \alpha}\left[\int_{D}(u[f+\alpha \hat{f}](x, T)-u[f](x, T))\left(u[f](x, T)+u[f+\alpha \hat{f}](x, T)-2 h^{\dot{\sigma}}(x)\right) d x\right] \\
			& +\lim _{\alpha \rightarrow 0} \frac{\gamma}{2 \alpha}\left[\int_{D} 2 \alpha \hat{f}\left(f-f^{0}\right)+(\alpha \hat{f})^{2} d x\right] \\
			& =\int_{D} \hat{u}[\hat{f}](x, T)\left(u[f](x, T)-h^{\delta}(x)\right) d x+\gamma \int_{D}\left(f-f^{0}\right) \hat{f} d x.
		\end{aligned}
		\tag{3.18}
		\label{eq:3.18}
	\end{equation}
	Introducing the adjoint equation of (\ref{eq:3.13}) as follows
	\begin{equation}
		\left\{\begin{aligned}
			-d P(x, t) & =A P(x, t) d t, & & (x, t) \in D \times(0, T], \\
			P(x, t) & =0, & & (x, t) \in \partial D \times(0, T], \\
			P(x, T) & =u[f](x, T)-h^{\delta,N}(x), & & x\in \bar{D},
		\end{aligned}\right.
		\tag{3.19}
		\label{eq:3.19}
	\end{equation}
	and substituting Eq. (\ref{eq:3.13}) and Eq. (\ref{eq:3.19}) into (\ref{eq:3.17}), it yields
	\begin{align*}
		\int_{D} \int_{0}^{T} P \hat{f} r(t) d t d x=\int_{0}^{T} \int_{D} \hat{u}[\hat{f}](x, t)\left(u[f](x, T)-h^{\delta,N}(x)\right) d x d t.
		\tag{3.20}
		\label{eq:3.20}
	\end{align*}
	By employing (\ref{eq:3.18}) and (\ref{eq:3.20}),  we deduce
	\begin{align*}
		\left(\nabla J_{r}[f], \widehat{f}\right)_{L^{2}(D)} & =\int_{D} \int_{0}^{T} P \widehat{f} R(t) d t d x+\gamma \int_{D} (f-f^{0}) \hat{f} d x \\
		& =\left(\left(\int_{0}^{T} P(x, t) R(t) d t+\gamma (f-f^{0})\right), \hat{f}\right)_{L^{2}(D)}.
	\end{align*}
	\par
	Therefore
	\begin{align*}
		\nabla J_{r}[f]=\int_{0}^{T} P(x, t) R(t) d t+\gamma (f(x)-f^{0}(x)).
		\tag{3.21}
		\label{eq:3.21}
	\end{align*}
	\section{The new weighting formula and variational inference}
    In the previous two sections, we analyze the properties of the solutions to Eq. (\ref{eq:2.1}) and find that directly solving (\ref{eq:2.12}) is unstable due to the ill-conditioned covariance matrix of (\ref{eq:2.12}). We propose a new weighting formula to address this problem and validate its rationality. Finally, we derive the optimality condition for the UQ based on variational inference and the proposed weighting formula.
	\subsection{The new weighting formula}
	In Section 2.2, we show that the Bayesian MAP estimate is equivalent to solving:
	\begin{align*}
		min \quad J=\frac{1}{2}\left\|h^{\delta, N}-F(f)\right\|_{\Sigma}^{2}+\frac{1}{2 N^{2}}\left\|f-f^{0}\right\|_{C_{0}}^{2}.
		\tag{4.1}
		\label{eq:4.1}
	\end{align*}
    \par
	To stably solve (4.1), we propose a new weighting formula as follows
	\begin{align*}
		\Sigma_{\beta}(x) = \left[\frac{\Sigma(x)}{h^{\delta,N}(x) + \sqrt{\Sigma(x)}}\right]^{\frac{1}{n^{*}}},
		\tag{4.2}
		\label{eq:4.2}
	\end{align*}
	where $\Sigma(x)$ denotes the original covariance value at point $x$, and $\frac{1}{n^{*}} = \left\lfloor\frac{(\operatorname{Cond}(\Sigma)-1) \alpha^{k}}{C_{1}}\right\rfloor$ defines the ill-posedness coefficient. $\lfloor\cdot\rfloor$ denotes the floor function (rounding down to the nearest integer). $\operatorname{Cond}(\Sigma)$ is the condition number of the covariance matrix $\Sigma$. $C_{1}$ is a hyperparameter associated with the observation position $x$ and $\alpha \in (0, 1)$ is a constant. Based on the proposed weighting formula (\ref{eq:4.2}), we reformulate (\ref{eq:4.1}) as
	\begin{align*}
		J_{new}=\frac{1}{2}\left\|h^{\delta, N}-F(f)\right\|_{\Sigma_{\beta}}^{2}+\frac{1}{2 N^{2}}\left\|f-f^{0}\right\|_{C_{0}}^{2}.
		\tag{4.3}
		\label{eq:4.3}
	\end{align*}
	\begin{remark}
		The proposed weighting formula combines the original covariance information with relative error, preserving the structural properties of the covariance matrix while capturing the scaling relationship between observed data and absolute error. To enhance the stability of solving (\ref{eq:4.3}), we introduce the iterations count $k$ and the condition number of the covariance matrix, ensuring convergence of the proposed weighting formula to the standard weighting formula for the i.i.d. case. This convergence is essential for deriving the condition for uncertainty quantification.
	\end{remark}
	\par
	To directly illustrate the relationship between the proposed weighting and the original formula. We assume that the observed data $h^{\delta,N}\left(x_{i}\right)$ and $h^{\delta,N}\left(x_{j}\right)$ are statistically independent for any $i \neq j$. The covariance matrix is simplified to a diagonal matrix, $\operatorname{diag}(\Sigma)=\{\sigma_{1}^{2},\sigma_{2}^{2},...,\sigma_{M}^{2}\}$. Therefore, (\ref{eq:4.1}) and (\ref{eq:4.3}) can be written as
	\begin{align*}
		J & =\frac{1}{2}\left\|h^{\delta,N}-F(f)\right\|_{\Sigma}^{2}+\frac{1}{2 N^{2}}\left\|f-f^{0}\right\|_{C_{0}}^{2} \\
		& =\frac{1}{2} \sum_{j=1}^{M}\left(\frac{1}{\sigma_{j}^{2}}\right)\left(h^{\delta,N}\left(x_{j}\right)-F(f)\left(x_{j}\right)\right)^{2}+\frac{1}{2 N^{2}}\|f-f^{0}\|_{C_{0}}^{2}.
		\tag{4.4}
		\label{eq:4.4}
	\end{align*}
	and
	\begin{align*}
		J_{new} & =\frac{1}{2}\left\|h^{\delta,N}-F(f)\right\|_{\Sigma_{\beta}}^{2}+\frac{1}{2 N^{2}}\left\|f-f^{0}\right\|_{C_{0}}^{2} \\
		& =\frac{1}{2} \sum_{j=1}^{M}\left(\frac{h^{\delta}\left(x_{j}\right)+\sigma_{j}}{\sigma_{j}^{2}}\right)^{\frac{1}{n^{*}}}\left(h^{\delta,N}\left(x_{j}\right)-F(f)\left(x_{j}\right)\right)^{2}+\frac{1}{2 N^{2}}\left\|f-f^{0}\right\|_{C_{0}}^{2}.
		\tag{4.5}
		\label{eq:4.5}
	\end{align*}
	When the observed data are i.i.d. ($\sigma_{j}=\sigma, j=1,...,M$), the (\ref{eq:4.4}) reduces to
	\begin{align*}
		J & =\frac{1}{2}\left\|h^{\delta, N}-F(f)\right\|_{\Sigma}^{2}+\frac{1}{2 N^{2}}\left\|f-f^{0}\right\|_{C_{0}}^{2} \\
		& =\frac{1}{2 \sigma^{2}}\left\|h^{\delta, N}-F(f)\right\|_{2}^{2}+\frac{1}{2 N^{2}}\left\|f-f^{0}\right\|_{C_{0}}^{2} \\
		& \propto \frac{1}{2} \sum_{j=1}^{M}\left(h^{\delta,N}\left(x_{j}\right)-F(f)\left(x_{j}\right)\right)^{2}+\frac{\sigma^{2}}{2 N^{2}}\left\|f-f^{0}\right\|_{C_{0}}^{2}.
		\tag{4.6}
		\label{eq:4.6}
	\end{align*}
	\par
	The primary difference among formulas (\ref{eq:4.4}), (\ref{eq:4.5}), and (\ref{eq:4.6}) lies in the weighting coefficient of the term $(h^{\delta,N}(x_j) - F(f)(x_j))^2$. For the weighting coefficient $(\frac{1}{\sigma_{j}^{2}})$ and $(\frac{h^{\delta}(x_{j})+\sigma_{j}}{\sigma_{j}^{2}} )^{2}$, we introduce two functions $ G(x)=\frac{1}{x^{2}}$ and $H(x)=\left[\frac{C+x}{x^{2}}\right]^{\frac{1}{n^{*}}} $, respectively, where $C$ represents the observed value $h^{\delta}(x_{j})$. Due to functions $H(x)$ and $G(x)$ sharing consistent monotonicity concerning $x$, both the proposed new weighting formula and the original formula exhibit identical trends in the variation of the covariance. Consequently, the proposed weighting formula preserves the essential characteristics of the original formula while incorporating additional features:
	\begin{itemize}
		\item As $C \rightarrow 0$ (small observed values), the growth rate of $H(x) \approx \left(\frac{1}{x}\right)^{\frac{1}{n^{*}}}$ is lower than $G(x) = \frac{1}{x^{2}}$. This mitigates the excessive growth for the weight of $(h^{\delta,N}(x_j) - F(f)(x_j))^2$ and the singular phenomena gradually disappears with the increase of $k$;
		\item As $m \ll C < C_{1}$ (large observed values), $\frac{1}{n^{*}} = \left\lfloor\frac{(\operatorname{cond}(\Sigma) - 1) \alpha^{k}}{C_{1}}\right\rfloor$ near to zero, implying $H(x) \approx \sqrt[n^{*}]{C}$ near to one. Thus, the error becomes negligible relative to $C$, allowing the observed data to be treated as i.i.d. data approximately;
		\item As $k$ increases, the iterative solution of (\ref{eq:4.3}) approaches the exact solution. It is demonstrated that the convergence of the proposed weighting formula to the standard covariance formula for the i.i.d. case (equivalent to $H(x) \rightarrow 1$);
		\item When the observed data are i.i.d., the weight function $ H (x) = 1$ is obviously found. In this case, the proposed weighting formula equals the standard covariance formula for the i.i.d. data.
	\end{itemize}
	\par
	In summary, the proposed weighting formula not only preserves the essential characteristics of the original covariance formula but also solves the (\ref{eq:4.3}) stably. As the number of iterations $k$ increases, the proposed formula converges to the standard covariance formula for i.i.d. observed data. This establishes a theoretical bridge between the i.i.d. and non-i.i.d. cases in studying stochastic inverse problems.
	\subsection{Variational inference}
    In the previous section, we use the new weighting formula to solve (\ref{eq:4.3}) stably. However, this process does not quantify the uncertainty of the solution. To achieve the UQ, the complete posterior distribution of the solution must be determined. Direct computation of the posterior distribution is computationally intractable, especially in high-dimensional inversion problems. Variational inference provides an efficient alternative by optimizing the posterior distribution, avoiding the challenges of direct computation of the posterior distribution (\cite{ref21}). Specifically, we construct a tractable family of approximate distributions to minimize the divergence between the true and approximate posterior. This approach accurately approximates the posterior distribution (\cite{ref18}). To this end, we define a variational family $Q$, parameterized by $\theta$, as the feasible region for optimal approximation (in this work, $Q$ is Gaussian; detailed reasons are provided in Section 3.2). The optimization objective is to find the best approximation $q(f ; \theta) \in Q$ such that
	\begin{align*}
		q^{*}(f ; \theta)=\underset{q(f ; \theta) \in Q}{\arg \min } K L\left[q(f ; \theta) \| p\left(f \mid h^{\delta,N}\right)\right],
		\tag{4.7}
		\label{eq:4.7}
	\end{align*}
	where the $KL$ divergence measures the discrepancy between two probability distributions and is defined as
	\begin{align*}
		K L\left[q(f ; \theta) \| p\left(f \mid h^{\delta,N}\right)\right] & =\int q(f ; \theta) \log \frac{q(f ; \theta)}{p\left(f \mid h^{\delta,N}\right)} d f  =E_{q(f ; \theta)}\left[\log q(f ; \theta)-\log p\left(f \mid h^{\delta,N}\right)\right].
		\tag{4.8}
		\label{eq:4.8}
	\end{align*}
	Applying Bayes's theorem, we can express (\ref{eq:4.8}) as
	\begin{center}
		$K L\left[q(f ; \theta) \| p\left(f \mid h^{\delta,N}\right)\right]=E_{q(f ; \theta)}\left[\log q(f ; \theta)-\log p\left(f, h^{\delta,N}\right)+\log p(D)\right]$,
	\end{center}
	from this expression, and derive
	\begin{center}
		$\log p(D)=K L\left[q(f ; \theta) \| p\left(f \mid h^{\delta,N}\right)\right]+E_{q(f ; \theta)}\left[\log p\left(f, h^{\delta,N}\right)-\log q(f ; \theta)\right]$.
	\end{center}
	Since $\log p(D)$ is constant, minimizing the $KL$ divergence is equivalent to maximizing the evidence lower bound (ELBO):
	\begin{align*}
		\min K L\left[q(f ; \theta) \| p\left(f \mid h^{\delta,N}\right)\right] \Leftrightarrow \max E_{q(f ; \theta)}\left[\log p\left(f, h^{\delta,N}\right)-\log q(f ; \theta)\right],
	\end{align*}
    or equivalently, 
    \begin{align*}
    \min \left\{E_{q(f ; \theta)}[\log q(f ; \theta)-\log p(f)]-E_{q(f ; \theta)}\left[\log p\left(h^{\delta,N} \mid f\right)\right]\right\}.
    \end{align*}
	Let $L=E_{q(f ; \theta)}[\log q(f ; \theta)-\log p(f)]-E_{q(f ; \theta)}\left[\log p\left(h^{\delta,N} \mid f\right)\right]$. Then it holds
	\begin{align*}
		L & =\underbrace{E_{q(f ; \theta)}[\log q(f ; \theta)-\log p(f)]}_{\text{model complexity loss}}-\underbrace{E_{q(f ; \theta)}\left[\log p\left(h^{\delta,N} \mid f\right)\right]}_{\text{log-likelihood}} \\
		& =\underbrace{E_{q(f ; \theta)}[\log q(f ; \theta)-\log p(f)]}_{\text{model complexity loss}}+\underbrace{E_{q(f ; \theta)} \frac{1}{2} \sum_{j=1}^{M}\left(\frac{h^{\delta,N}\left(x_{j}\right)-F(f)\left(x_{j}\right)}{\sigma_{j}}\right)^{2}}_{\text{observation data loss}},
		\tag{4.9}
		\label{eq:4.9}
	\end{align*}
	where the first term represents the model complexity loss and the second term corresponds to the expected log-likelihood. However, $L$ cannot directly take the derivative of the distribution $q(f;\theta)$. Next, we analyze this minimization process $L$ using the reparameterization trick (\cite{ref9}).
	\par
    First, we approximate the inversion object $f$ in a finite-dimensional space. Let the posterior distribution of $f$ be $\mathcal{N}\left(\mu, \Gamma_{post}\right)$, where $\mu \in \mathbb{R}^{M}$ and $\Gamma_{post} \in \mathbb{R}^{M \times M}$. To reparameterize $f \sim p(f; \theta)$, we introduce $\varepsilon \sim p(\varepsilon)$, which is independent of $\theta$ ($\mu$ and $H$ represent fundamental parameters of the $\theta$). We express $f$ as a deterministic function of $\theta$ and $\varepsilon$:
	\begin{align*}
		f=\mu +H\epsilon,
		\tag{4.10}
		\label{eq:4.10}
	\end{align*}
	where $H$ is a matrix and $\mu$ is the posterior mean of $f$. It is easy to compute $\Gamma_{post} = H^TH$, quantifying the uncertainty of $f$. To simplify the inference process (\ref{eq:4.9}), we assume that $H$ is a diagonal matrix: $\operatorname{diag}(H)=\{\sigma_{1}^{post},\sigma_{2}^{post},...,\sigma_{M}^{post}\}$, where $\sigma_{i}^{post} (i=1,\cdots ,M)$ are the posterior standard deviations of $f$ . 
	\begin{remark}
		Although the diagonal approximation of M may compromise the integrity of the UQ, it offers significant computational advantages.   Nevertheless, it is essential to emphasize that investigating M with richer structural characteristics holds significant potential and warrants further exploration.
	\end{remark}
	\par
	Consequently, $\Gamma_{post} = H^2$ also is a diagonal matrix. A simple computational process is as follows
	\begin{align*}
		q(f ; \theta) & =\frac{1}{(2 \pi)^{\frac{M}{2}} \operatorname{det}\left|\Gamma_{\text {post }}\right|} \exp \left(-\frac{1}{2}\|f-\mu\|_{\Gamma_{\text {post }}^{2}}^{2}\right) = \frac{1}{\operatorname{det}\left|\Gamma_{\text {post }}\right|} p(\varepsilon),
	\end{align*}
	where $p(\varepsilon)$ is the probability density function of $\mathcal{N}(0, I)$ ($I \in \mathbb{R}^{M \times M}$ denotes the identity matrix). Furthermore, let the prior distribution of $f$ be $f^{0} \sim \mathcal{N}(0, I)$, the model complexity loss term in (\ref{eq:4.9}) satisfies
	\begin{align*}
		E_{q(f ; \theta)}\{\log (q(f ; \theta))-\log (p(f))\} &=E_{q(f ; \theta)}\left\{\log \left(\frac{1}{\operatorname{det}\left|\Gamma_{\text {post }}\right|} p(\varepsilon)\right)-\log (p(\varepsilon))\right\} \\
		& =-\log \left(\operatorname{det}\left|\Gamma_{\text {post }}\right|\right).
		\tag{4.11}
		\label{eq:4.11}
	\end{align*}
	Substituting (\ref{eq:4.11}) into (\ref{eq:4.9}), it yields
	\begin{align*}
		L & =E_{q(f ; \theta)}\{\log (q(f ; \theta))-\log (p(f))\}+E_{q(f ; \theta)} \frac{1}{2} \sum_{j=1}^{M}\left(\frac{h^{\delta,N}\left(x_{j}\right)-F(f)\left(x_{j}\right)}{\sigma_{j}}\right)^{2} \\
		& =-\log \left(\operatorname{det}\left|\Gamma_{p o s t}\right|\right)+E_{q(f ; \theta)} \frac{1}{2} \sum_{j=1}^{M}\left(\frac{h^{\delta,N}\left(x_{j}\right)-F(f)\left(x_{j}\right)}{\sigma_{j}}\right)^{2} \\
		& =E_{q(f ; \theta)} \sum_{j=1}^{M}\left(-2\log \left(\sigma_{j}^{p o s t}\right)+\frac{1}{2}\left(\frac{h^{\delta,N}\left(x_{j}\right)-F(f)\left(x_{j}\right)}{\sigma_{j}}\right)^{2}\right).
		\tag{4.12}
		\label{eq:4.12}
	\end{align*}
    Since $(\frac{1}{\sigma_{j}^{2}})$ becomes ill-posed as $\sigma_{j} \rightarrow 0$, we use the proposed weighting formula to address this issue. Combining (\ref{eq:4.2}) with (\ref{eq:4.12}), we obtain
   \begin{align*}
   	L & \propto E_{q(f ; \theta)} \sum_{j=1}^{M}\left(\frac{\left(h^{\delta, N}\left(x_{j}\right)+\sigma_{j}\right)^{\frac{1}{n^{*}}}}{\left(\sigma_{j}\right)^{\frac{2}{n^{*}}-2}}\right)\left(-2\log \left(\sigma_{j}^{\text {post }}\right)+\frac{1}{2}\left(\frac{h^{\delta, N}\left(x_{j}\right)-F(f)\left(x_{j}\right)}{\sigma_{j}}\right)^{2}\right) \\
   	& =-\sum_{j=1}^{M}\left(\frac{h^{\delta, N}\left(x_{j}\right)+\sigma_{j}}{\sigma_{j}^{2}}\right)^{\frac{1}{n^{*}}}\left(2\sigma_{j}^{2}\right) \log \left(\sigma_{j}^{\text {post }}\right) \\
   	& +E_{q(f ; \theta)} \sum_{j=1}^{M}\left(\frac{h^{\delta, N}\left(x_{j}\right)+\sigma_{j}}{\sigma_{j}{ }^{2}}\right)^{\frac{1}{n^{*}}}\left(\frac{\left(h^{\delta, N}\left(x_{j}\right)-F(f)\left(x_{j}\right)\right)^{2}}{2}\right),
   	\tag{4.13}
   	\label{eq:4.13}
   \end{align*}
    where $h^{\delta, N}(x_{j})$ and $\sigma_{j}^{2}$ represent the mean and variance of $N$ observations at $x_{j}$, respectively. $E_{q(f ; \theta)}$ denotes integration over all $\theta$ in the parameter space. To  facilitate the calculation of (\ref{eq:4.13}), let's take  the optimization iterations $k$ large enough to make
	\begin{align*}
		L &{\approx}-\sum_{j=1}^{M}\left(2\sigma_{j}^{2} \log \left(\sigma_{j}^{\text {post }}\right)\right)+E_{q(f ; \theta)} \sum_{j=1}^{M}\left(\frac{\left(h^{\delta}\left(x_{j}\right)-F(f)\left(x_{j}\right)\right)^{2}}{2}\right) \\
		&=\underbrace{-\sum_{j=1}^{M} 2\sigma_{j}^{2} \log \left(\sigma_{j}^{\text {post }}\right)}_{\text {uncertainty regularization term}}+\overbrace{E_{q(f ; \theta)}}^{\text {sampling }}\underbrace{ \sum_{j=1}^{M}\left(\frac{\left(h^{\delta}\left(x_{j}\right)-F(f)\left(x_{j}\right)\right)^{2}}{2}\right)}_{\text {observation data loss}}.
		\tag{4.14}
		\label{eq:4.14}
	\end{align*}
	\par
	Therefore, solving the posterior distribution of $f$ is equivalent to optimizing $\mu$ and $\sigma_{j}^{post}$ to minimize the loss function $L$. An analysis of $L$ reveals the following insights: (1) The point estimate $\mu$ of $f$ is determined solely by the data loss term of (\ref{eq:4.14}), while the uncertainty quantification $\sigma_{j}^{post}$ depends on both the observation data loss and uncertainty regularization term. This provides a theoretical foundation for this two-stage optimization, indicating that the optimization target parameters $\mu$ and $\sigma_{j}^{post}$ can be optimized independently. (2) By combining (\ref{eq:4.11}) and (\ref{eq:4.14}), it can be observed that the mean of the prior distribution of $f$ has a minimal impact on the UQ. In contrast, its variance significantly affects the convergence rate. It offers guidance for setting the prior parameters. (3) (\ref{eq:4.14}) simplifies the uncertainty optimization condition from a limiting perspective, reducing computational costs, enhancing efficiency, and ensuring the reliable calculation of uncertainty.
	\section{Two-stage optimization process}
	This section describes the two-stage optimization method. In the first stage, the conjugate gradient optimization algorithm is employed to solve the modified regularization functional (\ref{eq:4.3}). In the second stage, we constructed the prior distribution of $f$, and a secondary optimization was performed using a random sampling method. Finally, the uncertainty of the solution is quantified using the prior distribution of $f$ and the uncertainty optimization formula (\ref{eq:4.13}).
	\subsection{The first-stage optimization process}
	We use the conjugate gradient optimization algorithm to solve the modified regularization functional (\ref{eq:4.3}) in the continuous sense: 
	\begin{align*}
		\begin{aligned}
			J_{\text {new }} \propto \frac{1}{2}\left\|\beta\left(h^{\delta,N}(x)-F(f)(x)\right)\right\|_{L^{2}(D)}^{2}+\frac{\gamma}{2}\left\|f(x)-f^{0}(x)\right\|_{L^{2}(D)}^{2},
		\end{aligned}
		\tag{5.1}
		\label{eq:5.1}
	\end{align*}
	where $\beta = \left[\frac{h^{\delta, N}(x) + (\Sigma(x))^{\frac{1}{2}}}{\Sigma(x)}\right]^{\frac{1}{h^{*}}}$ is the weighting coefficient and $\gamma$ is the regularization parameter. We employ Algorithm \ref{tab:1} to minimize (\ref{eq:5.1}). The main flow of Algorithm \ref{tab:1} is presented as follows:
	\begin{table}[ht]     
		\centering        
		\begin{tabular}{m{2.5cm} m{10.5cm}} 
			\toprule      
			\multicolumn{2}{l}{\textbf{Algorithm 1: Conjugate Gradient Optimization Algorithm}} \\ 
			\midrule      
			Step 1 & Initialize $f_{0}$ and set $k=0$; \\ 
			Step 2 & Calculate the weighted gradient $\nabla J_{\text {new }}[f_{0}]$, set $s_{0}=-\nabla J_{new}\left[f_{0}\right]$, $d_{0}=s_{0}$, $f^{0} \sim \mathcal{N}(0, I)$ and compute weighting coefficient $\beta_{0}$; \\
			Step 3 & Calculate the cost error $\varphi_{0}$, the intermediate quantity $v_{0}$, and the step length $\alpha_{0}$, where $\alpha_{0}=-\frac{\left\langle\beta_{0} \varphi, \beta_{0} v\left[d_{0}\right]\right\rangle_{L^{2}(D)}+\gamma\left\|f_{0} d_{0}\right\|_{L^{2}(D)}^{2}}{\left\|\beta_{0} v\left[d_{0}\right]\right\|_{L^{2}(D)}^{2}+\gamma\left\|d_{0}\right\|_{L^{2}(D)}^{2}}$; \\
			Start iteration & Set $f_{1}=f_{0}+\alpha_{0}d_{0}$ and $k=1,2, \cdots. $ \\  
			Step 4 & Compute $s_{k}=-\nabla J_{new}\left[f_{k}\right]$, $\zeta_{k}=\frac{\left\|s_{k}\right\|}{\left\|s_{k-1}\right\|}$, and $d_{k}=s_{k}+\zeta_{k} d_{k-1}$;\\
			Step 5 & Calculate the cost error $\varphi_{k}$, the intermediate quantity $v_{k}$, and the step length $\alpha_{k}$, where $\alpha _{k}=-\frac{\left\langle\beta_{k} \varphi_{k}, \beta_{k} v\left[d_{k}\right]\right\rangle_{L^{2}(D)}+\gamma\left\|f_{k} d_{k}\right\|_{L^{2}(D)}^{2}}{\left\|\beta_{k} v\left[d_{k}\right]\right\|_{L^{2}(D)}^{2}+\gamma\left\|d_{k}\right\|_{L^{2}(D)}^{2}}$;\\
			Step 6 & Update $f_{k+1}=f_{k}+a_{k} d_{k}$;\\
			Step 7 & Check the termination criteria;\\
			\bottomrule     
		\end{tabular}  
		\caption{first stage optimization: CGM for solving the weighted functional} 
		\label{tab:1} 
	\end{table}
    \par
	 To illustrate the key computational steps of the Algorithm \ref{tab:1}, we calculate the physical quantities involved in the $k$-th iteration. First, according to (\ref{eq:3.21}), the weighted gradient of the objective functional (\ref{eq:5.1}) is given by
	\begin{align*}
		\nabla J_{\text {new }}[f_{k}]=\int_{0}^{T} P_{k}(x, t) R(t) d t+\gamma (f_{k}-f^{0}),
		\tag{5.2}
		\label{eq:5.2}
	\end{align*}
	where $P_{k}(x,t)$ represents the adjoint state variable, which is the solution of Eq. (\ref{eq:3.19}), satisfying
	\begin{equation}
		\left\{\begin{aligned}
			-d P_{k}(x, t) & =A P_{k}(x, t) d t, & & (x, t) \in D \times(0, T], \\
			P_{k}(x, t) & =0, & & (x, t) \in \partial D \times(0, T], \\
			P_{k}(x, T) & =u[f_{k}](x, T)-h^{\delta,N}(x), & &  x\in \bar{D}.
		\end{aligned}\right.
		\tag{5.3}
		\label{eq:5.3}
	\end{equation}
	Thus, Step 2 and Step 4 in Algorithm \ref{tab:1} are completed. Here, the physical quantity $u[f_{k}](x, T)$ in Eq. (\ref{eq:5.3}) is the solution of Eq. (\ref{eq:3.13}) at $t=T$, as follows
	\begin{equation} 
		\left\{\begin{aligned}
			d u[f_{k}](x,t) & =A u[f_{k}](x,t) d t+f_{k}(x) R(t) d t, & & (x, t) \in D \times(0, T], \\
			u[f_{k}](x,t) & =0, & & (x, t) \in \partial D \times(0, T], \\
			u[f_{k}](x,t) & =0, & &  x\in \bar{D}, t=0.
		\end{aligned}\right.
		\tag{5.4}
		\label{eq:5.4}
	\end{equation}
	To implement Step 3 and Step 5 in Algorithm \ref{tab:1}, we compute the intermediate physical quantity $v_{k}$, which is the solution of the following equation
	\begin{equation}
		\left\{\begin{aligned}
			d v\left[f_{k}\right](x,t) & =A v\left[f_{k}\right](x,t) d t+d_{k} d t, & & (x, t) \in D \times(0, T], \\
			v\left[f_{k}\right](x,t) & =0, & & (x, t) \in \partial D \times(0, T], \\
			v\left[f_{k}\right](x,t) & =0, & &  x\in \bar{D}, t=0,
		\end{aligned}\right.
		\tag{5.5}
		\label{eq:5.5}
	\end{equation}
    where $d_{k} = u[f_{k}](x, T)-h^{\delta,N}(x)$ denotes the residual difference between the observed data $h^{\delta,N}(x)$ and the solution $u[f_{k}](x, t)$ of Eq. (\ref{eq:5.4}) at time $t=T$.

	\subsection{The second-stage optimization process}
	After completing the first optimization stage, we obtain the preliminary optimal solution $f_{0}^{*}$ of the regularization functional, approximating the exact solution $f$.   We  use the random sampling method to secondary optimization of $f_{0}^{*}$, followed by UQ based on the condition $L$:
	\begin{align*}
		L =-\sum_{j=1}^{M} 2 \sigma_{j}^{2} \log \left(\sigma_{j}^{\text {post }}\right)+\frac{1}{2} \sum_{j=1}^{M}\left(h^{\delta,N}\left(x_{j}\right)-F(f)\left(x_{j}\right)\right)^{2}.
		\tag{5.6}
		\label{eq:5.6}
	\end{align*}
	Similar to the computation in (\ref{eq:5.1}), we first derive the optimization gradient of $L$ with respect to $\sigma_{j}^{\text {post}}$. By differentiating  $L$ in (\ref{eq:5.6}) with respect to $\sigma_{j}^{\text {post }}$, we obtain
	\begin{align*}
		\frac{\partial L}{\partial \sigma_{j}^{\text {post }}}=-\frac{2 \sigma_{j}^{2}}{\sigma_{j}^{\text {post }}} +\nabla J_{\text {new }}^{*}[f] \frac{\partial f}{\partial H_{j}},
	\end{align*}
    where $\sigma_{j}^{\text{post}}$ represents the $j$-th optimized variable, and $\sigma_{j}$ denotes the $j$-th element of the covariance matrix obtained from the observed data. Combining (\ref{eq:4.10}) with the first-order optimization gradient (\ref{eq:5.2}), we derive the gradient of $H$ as
	\begin{align*}
		\nabla H=\left[\frac{\partial L}{\partial \sigma_{j}^{\text {post }}}\right]_{M \times 1}=\left[-\frac{2 \sigma_{j}^{2}}{\sigma_{j}^{\text {post }}} +\nabla J_{\text {new }}^{*}[f] \epsilon _{j}\right]_{M \times 1},
		\tag{5.8}
		\label{eq:5.8}
	\end{align*}
	where $\nabla J_{new}^{*}\left[f_{j}\right]$ is the gradient at $x_{j}$ given by
	\begin{center}
		$\nabla J_{new}^{*}(x_{j})=\int_{0}^{T} P(x_{j}, t) R(t) dt$,
	\end{center}
	and $\epsilon_{j}$ denotes the $j$-th randomly generated number in (\ref{eq:4.10}).
	\par
	Therefore, the main flow of Algorithm \ref{tab:2} is presented as follows:
	\begin{table}[H]     
		\centering        
		\begin{tabular}{m{2.3cm} m{11.5cm}} 
			\toprule      
			\multicolumn{2}{l}{\textbf{Algorithm 2: Random Sampling and Gradient Descent Algorithm}} \\ 
			\midrule      
			Step 1 & Set the uncertainty upper limit $H_{0}$, and the optimal solution $f_{0}^{*}$; \\ 
			Step 2 & Generate a sequence based on the Gaussian prior: $f_{i+1}^{*}=H_{i} \varepsilon_{i}+f_{0}^{*}$, $i=0,\cdots,N$; \\
			Step 3 & Optimize $f^{*}$ use the random sampling method, $f_{best}^{*}=\arg \min \left\{L^{*}\left(f_{i}^{*}\right), L^{*}\left(f_{i+1}^{*}\right)\right\}$, where the loss function $L^{*}$ is defined as: $L^{*}(f_{i}^{*})=\sum_{j=1}^{M}\left(h^{\delta,N}\left(x_{j}\right)-F(f_{i}^{*})\left(x_{j}\right)\right)^{2}$; \\
			Start iteration & $k=1,2, \cdots, N^{*}$, and $d$ is the gradient descent step size; \\  
			Step 4 & Calculate $H_{k+1}=-\nabla H_{k} d+H_{k}$ and $f_{k+1}^{*}=H_{k} \varepsilon_{k}+f_{\text {best }}^{*}$;\\
			Step 5 & Determine $H_{best}=\arg \min \left\{L^{*}\left(f_{k+1}^{*}\right), L^{*}\left(f_{\text {best }}^{*}\right)\right\}$, ensuring the largest possible range without lose of accuracy.\\
			\bottomrule     
		\end{tabular}  
		\caption{second stage optimization: random sampling and uncertainty quantification} 
		\label{tab:2} 
	\end{table} 
	\begin{remark}
		Based on the computations from the first stage, we have obtained a reliable approximation of the inversion target $f$. Building on this, we further refine $f$ through the random sampling method. This optimization strategy significantly reduces computational costs and effectively shortens the burn-in time of data sampling. For specific implementation details of the random sampling algorithm, please refer to \cite{ref20}.
	\end{remark}
	\par
	    Finally, based on the optimal estimation results of the regularized solution presented in Section 3, we establish a theorem for error estimation that relates the exact solution to the regularized solution under $N^2$ observations. 
	\begin{theorem}\label{th:5.1}
		Under the assumptions of Theorem \ref{th:3.2}, consider the regularization functional (\ref{eq:5.1}) with $M = \|g\|_{L^{2}(D)}$ and $0 < \beta \le \beta_{max}$. There exist constants $C>0$ such that
		\begin{align*}
			\mathbb{E} \left[ \left\| f(x) - f_{\gamma}^{\delta}(x) \right\|_{L^{2}(D)}^2 \right] \leq C\left ( \frac{M}{N} \right )^{\frac{4}{3}} ,
		\end{align*}
		where $f(x)$ is the inversion target in (\ref{eq:2.1}), $f_{\gamma}^{\delta}(x)$ is the reconstructed solution, and $N$ is the number of observations. Here, $C = \left( \frac{2K^2 C_1^2 C_{\lambda}^{\frac{2}{3}}}{\beta_{\max}^2(C_R C_0)^4} + \frac{\beta_{\max}C_{\lambda}^{\frac{2}{3}}}{2C_1} \right)$ with $C_\lambda = \frac{1}{2\lambda_{1}}$, $K=\left \| f(x)-f^{0}(x) \right \|_{H^2(D)} $, $C_0 = 1 - e^{-\lambda_1 T}$, $C_1>0$ is a constant, and $0 < C_R \leq R(t)$ on $[0,T]$.
	\end{theorem}
	\begin{proof}
		Define the modified regularization functional:
		\begin{align*}
			J_{\text{new}} 
			&= \frac{1}{2}\left\|\beta\left(h^{\delta, N}(x) - F(f)(x)\right)\right\|_{L^{2}(D)}^{2} + \frac{\gamma}{2}\left\|f(x)-f^{0}(x)\right\|_{L^{2}(D)}^{2} \\
			&\leq \frac{\beta_{\max}^{2}}{2}\left\|h^{\delta, N}(x) - F(f)(x)\right\|_{L^{2}(D)}^{2} + \frac{\gamma}{2}\left\|f(x)-f^{0}(x)\right\|_{L^{2}(D)}^{2}.
		\end{align*}
		This is equivalent to minimizing a functional proportional to
		\begin{align*}
			\frac{1}{2}\left\|h^{\delta, N}(x) - F(f)(x)\right\|_{L^{2}(D)}^{2} + \frac{\gamma_{1}}{2}\left\|f(x)-f^{0}(x)\right\|_{L^{2}(D)}^{2}.
		\end{align*}
        where $\gamma_{1}=\frac{C_{1} \delta}{\beta_{\max }}$. By Theorem \ref{th:3.2} with $\delta=\frac{\sqrt{C_\lambda}M}{N}$, and choosing $\gamma_{1}=\frac{C_{1} \delta^{\frac{2}{3}}}{\beta_{\max }}$, the error estimate is
	    \begin{align*}
	    	\mathbb{E}\left[\|f - f_{\gamma}^{\delta}\|_{L^2}^2\right] &\leq \left( \frac{2K^2 C_1^2}{\beta_{\max}^2(C_R C_0)^4} + \frac{\beta_{\max}}{2C_1} \right) \delta^{\frac{4}{3}} \\
	    	&= \left( \frac{2K^2 C_1^2}{\beta_{\max}^2(C_R C_0)^4} + \frac{\beta_{\max}}{2C_1} \right) \frac{C_{\lambda}^{\frac{2}{3}}M^{\frac{4}{3}}}{N^{\frac{4}{3}}}.
	    \end{align*}
		Let $C = \left( \frac{2K^2 C_1^2 C_{\lambda}^{\frac{2}{3}}}{\beta_{\max}^2(C_R C_0)^4} + \frac{\beta_{\max}C_{\lambda}^{\frac{2}{3}}}{2C_1} \right)$, which completes the proof.
	\end{proof}
	\begin{remark}
	    The error between the inverse and the exact solutions decreases monotonically with the number of observations $N$ while exhibiting significant growth with the noise level $M$ of the noise function $g(x)$. These results align rigorously with theoretical predictions and are further corroborated by comprehensive numerical simulations. 
    \end{remark}
	
	\section{Numerical Experiments}
	In our numerical experiments, we validate the proposed method using two examples. The inversion functions include both smooth functions and those with non-differentiable points. We employ the finite element method in all the experiments to numerically solve the SPDEs (see \cite{ref14,ref19} for details). The computational domain is $[0, \pi] \times [0,1]$, with observed data $T=1$, and discretization parameters $\triangle x = \frac{\pi}{100}$ and $\triangle t = \frac{1}{20}$. The expectation is approximated by the average of $N$ realizations, where $N$ is specified in the following examples.
	\begin{example}
		Use this smooth function as a representative example,  
		\begin{equation}
			\left\{\begin{aligned}
				d u & =\Delta u d t-x u d t+(2+x) e^{t} \sin x d t+x d w, & & (x, t) \in[0, \pi] \times(0, T], \\
				u & =0, & & (x, t) \in\{0, \pi\} \times(0, T], \\
				u_{0} & =\sin x, & & (x, t) \in[0, \pi] \times\{t=0\}.
			\end{aligned}\right.
			\tag{6.1}
			\label{eq:6.1}
		\end{equation}
    Define the inversion target as $f(x) = (2 + x) \sin(x)$. We examine the effectiveness of the proposed method, the influence of different $N$ on the reconstructions, and the similarities and differences between the proposed formula and the i.i.d. formula in their reconstruction outcomes.
	\end{example}
    \par
     In Figure \ref{fig:1}, Panel (a) illustrates the distribution of the mild solution in the expected sense, while Panel (b) depicts the corresponding standard deviation. The solutions at different positions follow distinct distributions. Specifically, when $x \in \partial D$, $\varphi_{n}(x) = 0$, resulting in a singular covariance matrix for the observed data. This observation is consistent with our theoretical analysis. Panel (c) shows the standard deviation distribution of one set of observed data, revealing that random noise is more complex than deterministic noise. Panels (d), (e), and (f) display the solutions, relative error, and absolute error, respectively, for multiple sets of observed data at time $t=1$. Notably, the fluctuation trends of these two errors are inconsistent. In particular, in the boundary region and near the central peak, points with more significant fluctuations are assigned incorrectly larger weights. This inconsistency directly affects the stability of the Bayesian MAP solution. Nevertheless, our proposed weighting formula effectively resolves this issue. The following section provides a detailed comparison and analysis of the inversion results to support our claims.
     \begin{figure}[H]
    	\centering
    	\subfigure[expectation-gentle solution]{\includegraphics[width=0.29\textwidth]{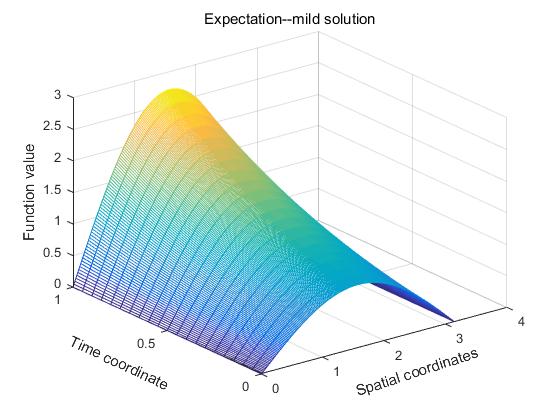}}
    	\subfigure[std-mild solution]{\includegraphics[width=0.29\textwidth]{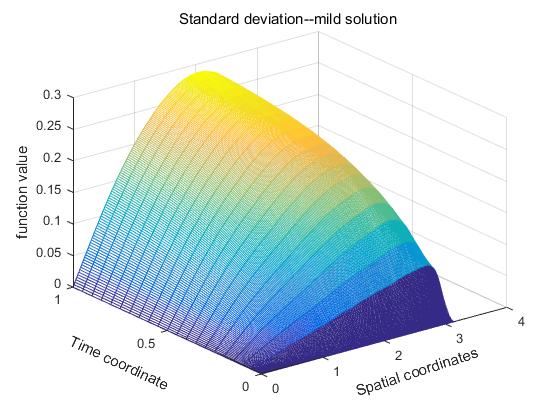}}
    	\subfigure[relative error-1 observation]{\includegraphics[width=0.29\textwidth]{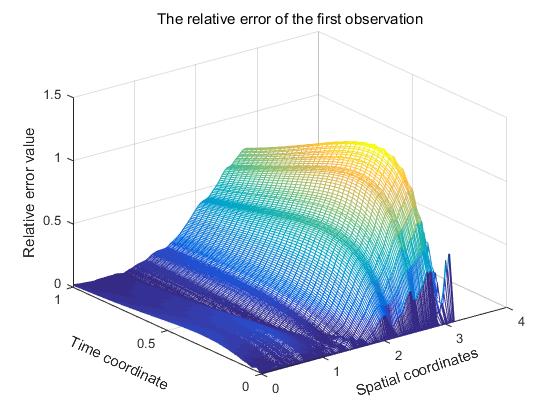}} \\
    	\subfigure[$t=1$, N observations]{\includegraphics[width=0.29\textwidth]{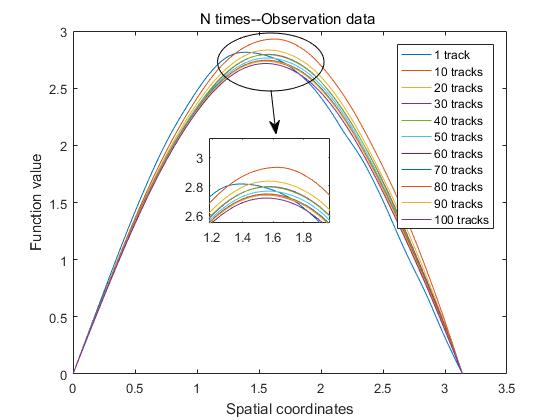}}	
    	\subfigure[$t=1$, absolute error]{\includegraphics[width=0.29\textwidth]{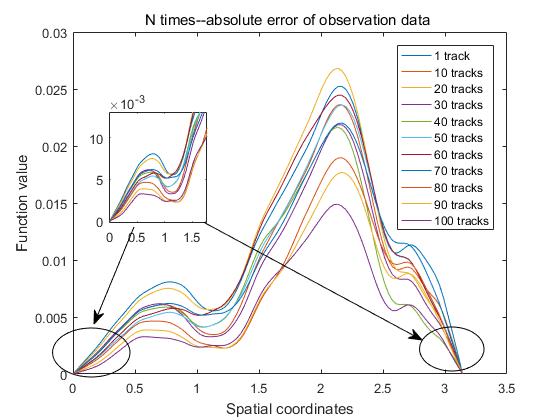}}	
    	\subfigure[$t=1$, relative error]{\includegraphics[width=0.29\textwidth]{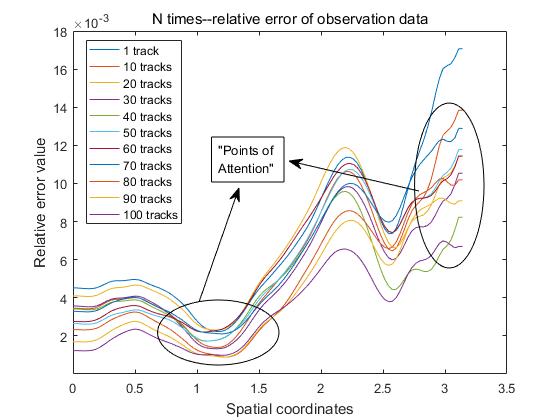}}	
    	\caption{Information of the mild solution to Eq. (\ref{eq:6.1}).}
    	\label{fig:1}
    \end{figure}
	\par
    In Figures \ref{fig:2} to \ref{fig:4}, we present results from source term inversion experiments conducted under varying levels of unknown noise (0\%, 1\%, 5\%, and 10\%). At the 0\% noise level, the basic structure of the inversion object is effectively reconstructed with five observations. In comparison, a total of 20 observations yields a more accurate estimation. However, as the noise level increases, the model's inversion capability deteriorates significantly, with no substantial improvement noted even when more observed data are added. This issue is especially pronounced at the 10\% noise level, where the expected observed data can cause significant distortions in the inversion results. The high level of unknown noise renders the observed data unreliable, leading to substantial deviations and ultimately making the problem unsolvable, as shown in Panel (f) of Figure \ref{fig:4}.
    \begin{figure}[H]
    	\centering
    	\subfigure[1 track-inversion]{\includegraphics[width=0.29\textwidth]{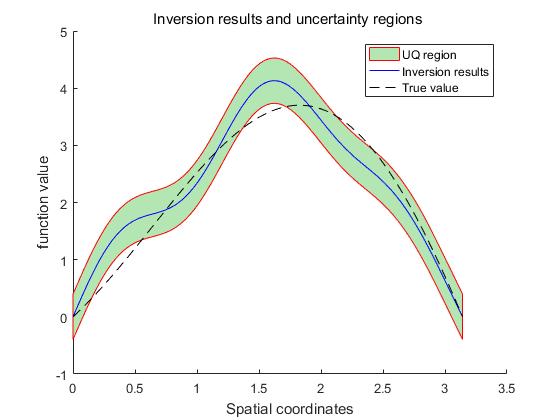}}
    	\subfigure[5 track-inversion]{\includegraphics[width=0.29\textwidth]{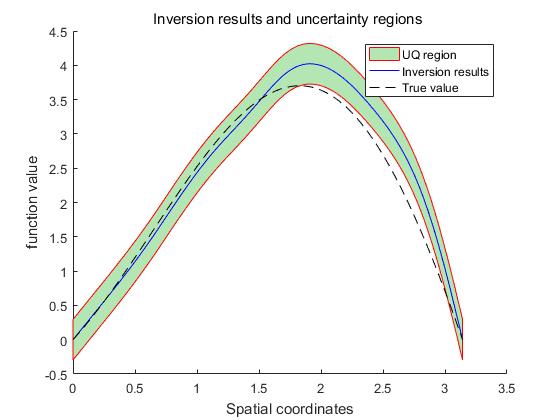}}
    	\subfigure[20 track-inversion]{\includegraphics[width=0.29\textwidth]{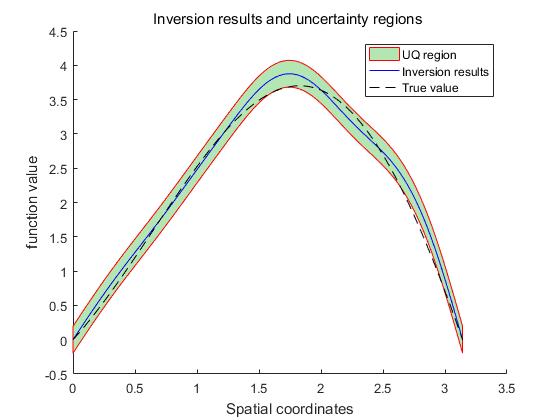}} \\
    	\subfigure[40 track-inversion]{\includegraphics[width=0.29\textwidth]{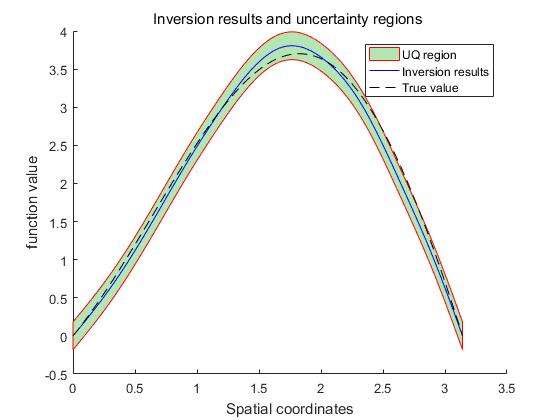}}	
    	\subfigure[300 track-inversion]{\includegraphics[width=0.29\textwidth]{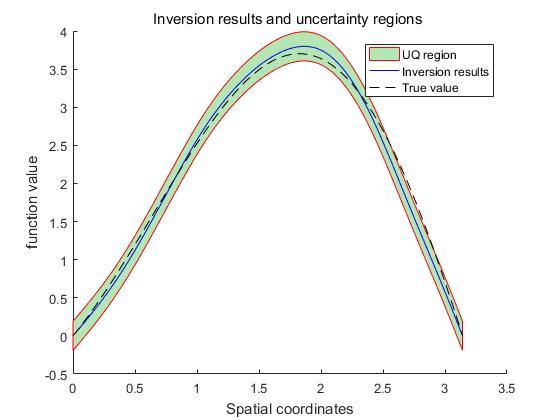}}	
    	\subfigure[expectation-inversion]{\includegraphics[width=0.29\textwidth]{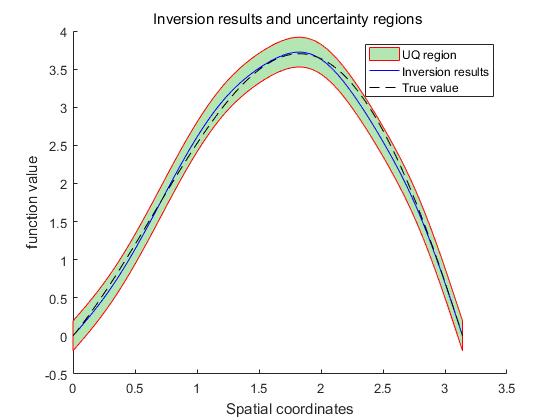}}
    	\caption{Comparison of inversion results under the condition of 0\% noise at time $t=1$.}
    	\label{fig:2}
    \end{figure}
    \begin{figure}[H]
    	\centering
    	\subfigure[10 track-inversion]{\includegraphics[width=0.29\textwidth]{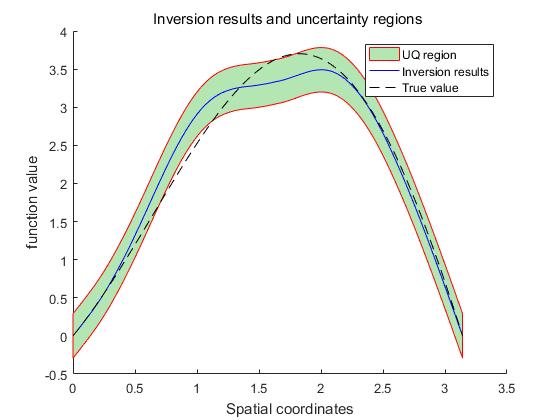}}
    	\subfigure[20 track-inversion]{\includegraphics[width=0.29\textwidth]{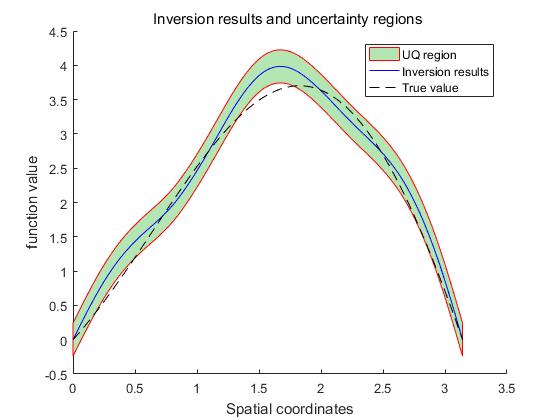}}
    	\subfigure[40 track-inversion]{\includegraphics[width=0.29\textwidth]{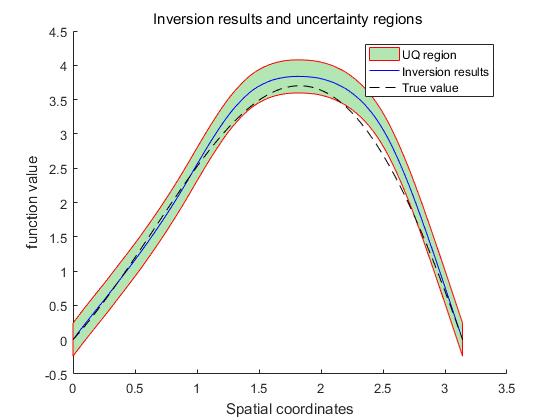}} \\
    	\subfigure[100 track-inversion]{\includegraphics[width=0.29\textwidth]{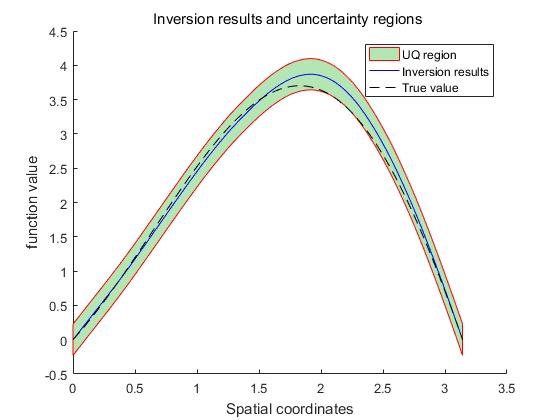}}	
    	\subfigure[300 track-inversion]{\includegraphics[width=0.29\textwidth]{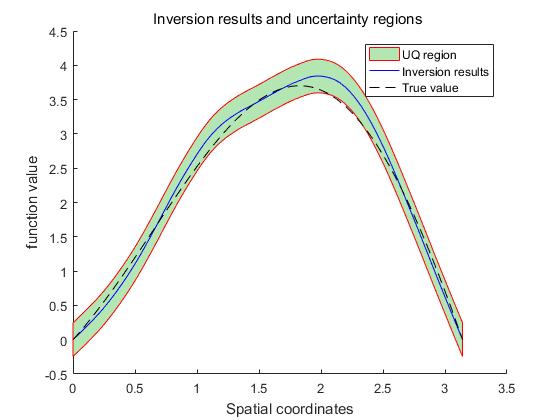}}	
    	\subfigure[expectation-inversion]{\includegraphics[width=0.29\textwidth]{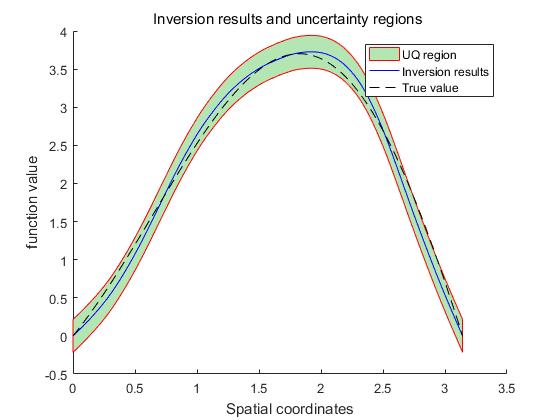}}	
    	\caption{Comparison of inversion results under the condition of 1\% noise at time $t=1$.}
    	\label{fig:3}
    \end{figure}
    \begin{figure}[H]
    	\centering
    	\subfigure[20 track-inversion]{\includegraphics[width=0.29\textwidth]{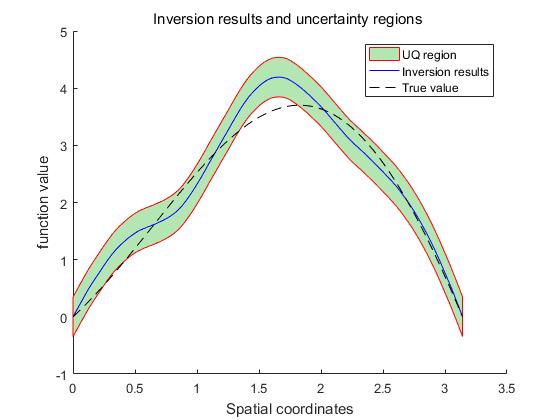}}
    	\subfigure[40 track-inversion]{\includegraphics[width=0.29\textwidth]{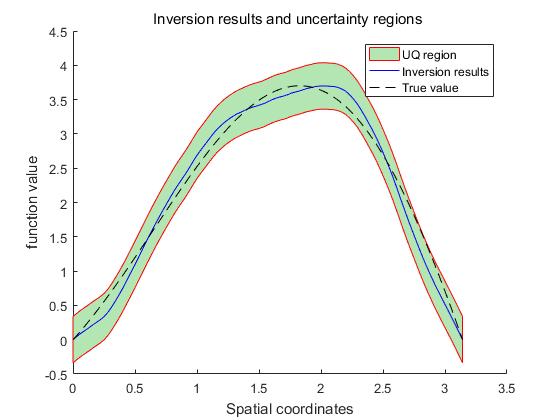}}
    	\subfigure[100 track-inversion]{\includegraphics[width=0.29\textwidth]{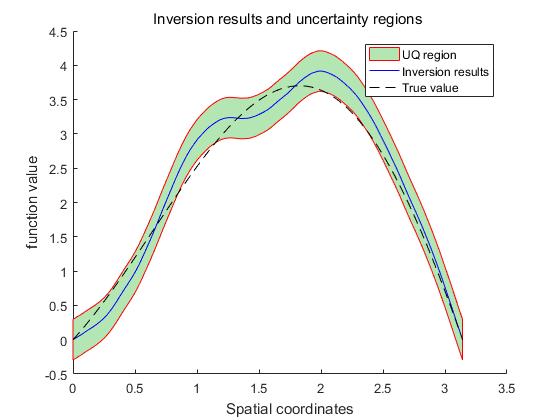}} \\
    	\subfigure[300 track-inversion]{\includegraphics[width=0.3\textwidth]{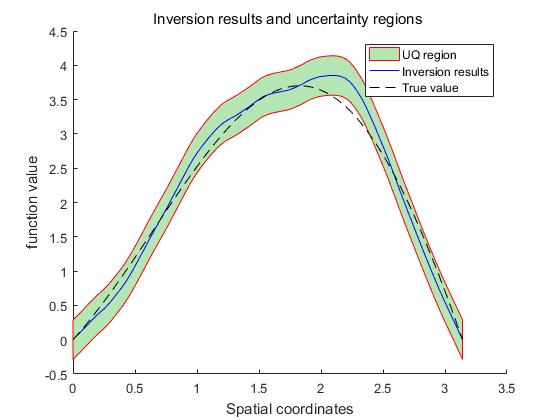}}	
    	\subfigure[5\% expectation-inversion]{\includegraphics[width=0.29\textwidth]{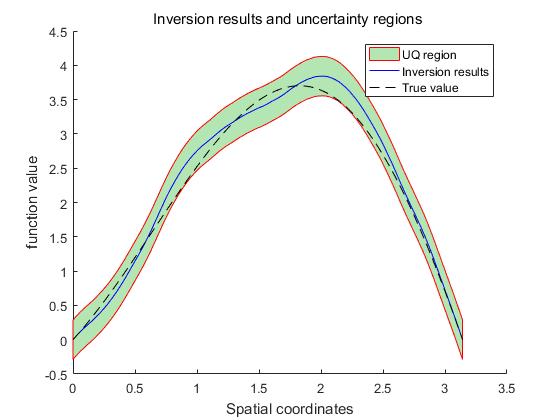}}
    	\subfigure[10\% expectation-inversion]{\includegraphics[width=0.29\textwidth]{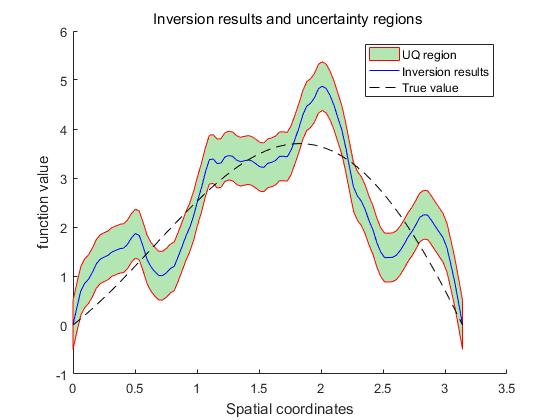}}
    	\caption{Comparison of inversion results under 5\% and 10\% noise conditions at time $t=1$.}
    	\label{fig:4}
    \end{figure}
    \par 
    On the other hand, Figures \ref{fig:2} to \ref{fig:4} further illustrate the quantification results of the uncertainty regions. The results demonstrate that the uncertainty regions generally encompass the ranges of inversion error. Although full coverage is not achieved in specific ranges, this still supports the validity of our proposed uncertainty quantification formula. Furthermore, for a fixed level of unknown noise, as the number of observations increases and random noise decreases, the uncertainty region progressively shrinks, which is consistent with our theoretical framework. However, significant distortion occurs in the optimal estimation when the unknown noise exceeds a certain threshold (e.g., 10\%). Although the uncertainty region can still be quantified, the inversion results become unreliable and lose practical significance. Therefore, quantifying the uncertainty region depends not only on the random noise level and the number of observations but also on the intensity of unknown noise and the precision of the optimal estimation solution, as stated in Theorem \ref{th:5.1}. See Tables \ref{tab:3} and \ref{tab:4} for specific numerical relationships. Here, $\left \| f \right \|_{\infty}$ represents the maximum absolute error in the inversion result. $UQ_{max}$ and $UQ_{min}$ respectively represent the upper and lower bounds of the UQ region of the inversion result.
    	\begin{table}[H]
    	\begin{center}
    		\begin{tabular}{c|cccc}\toprule
    			Observation (N) & 1 track  & 5 tracks & 10 tracks & 20 tracks \\
    			\hline
    			$\left \| f \right \|_{\infty}$ & $0.3997$ & $0.2092$ & $0.1843$ & $0.1825$\\
    			$UQ_{max}$ & $0.3971$ & $0.1957$ & $0.1913$ & $0.1874$\\
    			$UQ_{min}$ & $0.3963$ & $0.1954$ & $0.1909$ & $0.1869$\\
    			\hline
    			Observation (N) & 40 tracks  & 100 tracks & 300 tracks & Expectation \\
    			\hline
    			$\left \| f \right \|_{\infty}$ & $0.1807$ & $0.1545$ & $0.1361$ & $0.1312$\\
    			$UQ_{max}$ & $0.1827$ & $0.1815$ & $0.1781$ & $0.1725$\\
    			$UQ_{min}$ & $0.1819$ & $0.1809$ & $0.1778$ & $0.1721$\\
    			\bottomrule
    		\end{tabular}
    		\caption{The comparison of inversion error and uncertainty width of equation (6.1) under 0\% noise.}
    		\label{tab:3}
    	\end{center}
    \end{table}

    \begin{table}[H]
    	\begin{center}
    		\begin{tabular}{c|ccccc}
    			\toprule
    			\multicolumn{6}{c}{\textbf{Comparison of inversion results of $\left \| f \right \|_{\infty}$ under different noise levels, ($1\times 10^{-1}$)}}\\
    			\hline
    			Observation (N) & 10 track  & 20 tracks & 40 tracks & 100 tracks & 300 tracks \\
    			\hline
    			$0\%$ & $1.843$ & $1.825$ & $1.807$ & $1.545$ & $1.361$\\
    			$1\%$ & $2.539$ & $2.421$ & $2.196$ & $2.011$ & $1.844$\\
    			$5\%$ & $3.471$ & $3.227$ & $2.854$ & $2.705$ & $2.457$\\
    			$10\%$& $--$ & $--$ & $--$ & $--$ & $12.769$\\
    			\hline
    			\multicolumn{6}{c}{\raggedright \textbf{Uncertainty width ($UQ_{min}$, $UQ_{max}$) under different noise levels, ($1\times 10^{-1}$)}} \\
    			\hline
    			Observation (N) & 10 track  & 20 tracks & 40 tracks & 100 tracks & 300 tracks \\
    			\hline
    			$0\%$ & $(1.909,1.913)$ & $(1.869,1.874)$ & $(1.819,1.827)$ & $(1.809,1.815)$ & $(1.778,1.781)$\\
    			$1\%$ & $(2.635,2.639)$ & $(2.413,2.417)$ & $(2.225,2.227)$ & $(1.954,1.956)$ & $(1.909,1.912)$\\
    			$5\%$ & $(3.459,3.463)$ & $(3.381,3.387)$ & $(2.914,2.917)$ & $(2.711,2.716)$ & $(2.454,2.455)$\\
    			\bottomrule
    		\end{tabular}
    		\caption{Compare the inversion error and uncertainty width of Eq. (\ref{eq:6.1}) under different noise conditions.}
    		\label{tab:4}
    	\end{center}
    \end{table}
    \par
    Finally, we compare the performance of the proposed formula with the i.i.d. formula using the same observed data. Figure \ref{fig:5} presents the inversion results obtained using the proposed formula (dashed line) and the i.i.d. formula (solid line). (Note that the unimproved Bayesian MAP formula in (\ref{eq:4.1}) fails to compute and is therefore not shown.) Figure \ref{fig:5} shows that when the unknown noise level is below 5\%, the proposed formula outperforms the i.i.d. formula. This improvement can be attributed to the proposed formula's enhanced capability to balance the issue of small covariance at the data ends while accounting for significant relative errors. However, as the unknown noise level increases, the performance of the proposed formula becomes comparable to that of the i.i.d. formula.
    In contrast, the i.i.d. formula exhibits more excellent stability. This is because the weight distribution in the proposed formula does not explicitly account for the unknown noise, leading to deviations from the actual situation. As the unknown noise level rises, these deviations become more pronounced. 
    In contrast, the i.i.d. formula assumes equal credibility for all data, ensuring its computational stability is unaffected by unknown noise. This is the primary reason we propose that the new weight formula should converge to the i.i.d. formula after a sufficient number of iterations. Therefore, we conclude that the new formula suits Bayesian inversion problems under complex stochastic noise conditions and enhances stability in solving stochastic inverse problems.
    \begin{figure}[H]
    	\centering
    	\subfigure[0\%- different track-inversion]{\includegraphics[width=0.35\textwidth]{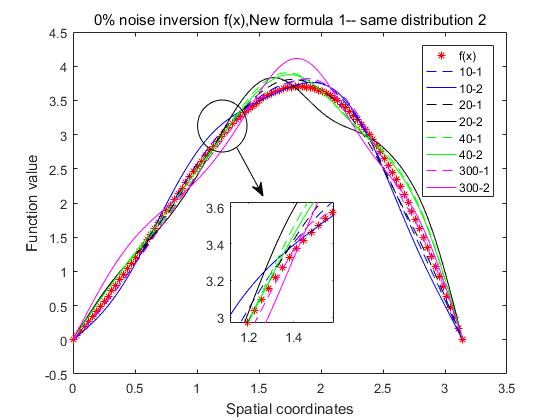}}
    	\subfigure[0\%-absolute error of inversion]{\includegraphics[width=0.35\textwidth]{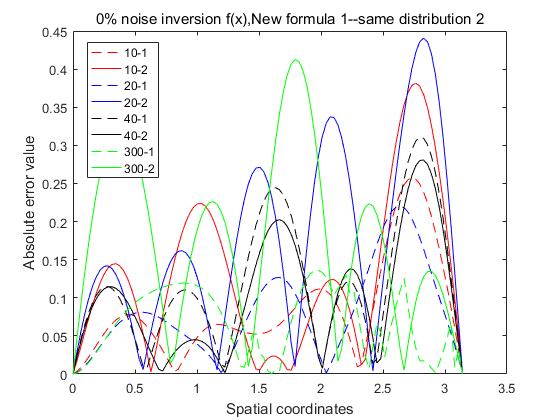}} \\
    	\subfigure[E-inversion]{\includegraphics[width=0.35\textwidth]{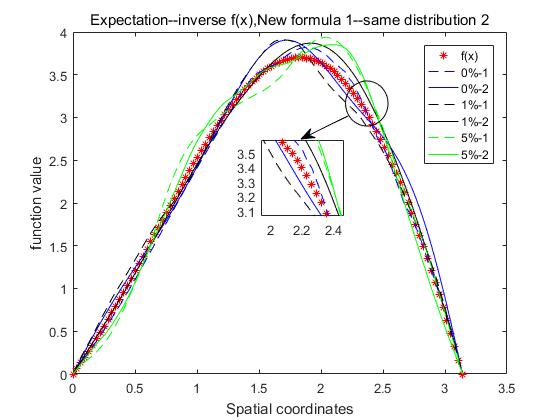}} 
    	\subfigure[E-absolute error of inversion]{\includegraphics[width=0.35\textwidth]{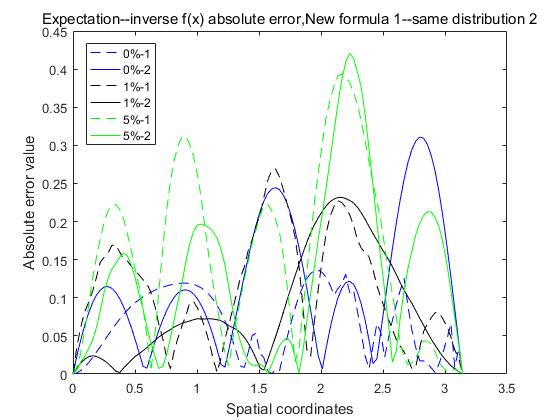}}
    	\caption{Comparison of inversion results of two kinds of formulas.}
    	\label{fig:5}
    \end{figure}
	\begin{example}
		Use this non-smooth function as a representative example,
		\begin{equation}
			\left\{\begin{aligned}
				d u & =\Delta u d t+e^{t} f(x) d t+0.5dw, & & (x, t) \in[0, \pi] \times(0, T], \\
				u & =0, & & (x, t) \in\{0, \pi\} \times(0, T], \\
				u_{0} & =0, & & (x, t) \in[0, \pi] \times\{t=0\}.
			\end{aligned}\right.
			\tag{6.2}
			\label{eq:6.2}
		\end{equation}
		The inversion target is  
		\begin{equation}
			f(x)=\left\{\begin{array}{ll}
				\quad	\frac{x}{2}, & x \in\left[0, \frac{\pi}{3}\right], \\
				\quad	\frac{\pi}{6}, & x \in\left(\frac{\pi}{3}, \frac{2 \pi}{3}\right], \\
				-\frac{x}{2}+\frac{\pi}{2}, & x \in\left(\frac{2\pi}{3}, \pi\right],
			\end{array}\right.
			\tag{6.3}
			\label{eq:6.3}
		\end{equation}
		we again examine the effectiveness of the proposed method and the influence of different $N$ on the reconstructions. Additionally, we compare the differences between Eq. (\ref{eq:6.2}) and (\ref{eq:6.1}) and analyze their impact on the inversion results.
	\end{example}
	\begin{figure}[H]
		\centering
		\subfigure[expectation-gentle solution]{\includegraphics[width=0.29\textwidth]{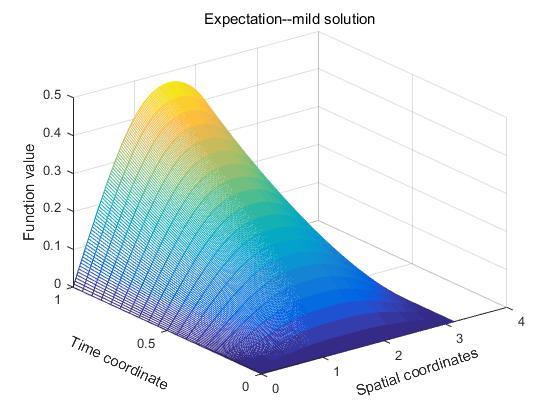}}
		\subfigure[std-mild solution]{\includegraphics[width=0.29\textwidth]{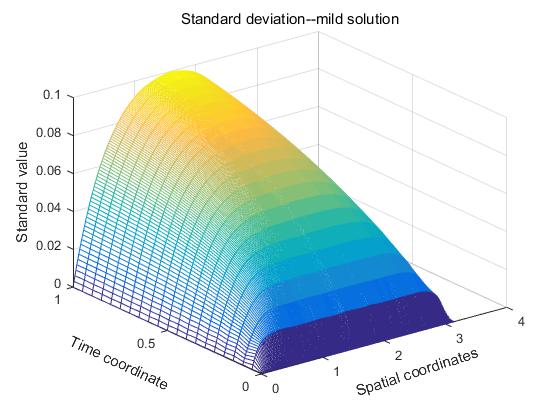}}
		\subfigure[relative error-1 observation]{\includegraphics[width=0.29\textwidth]{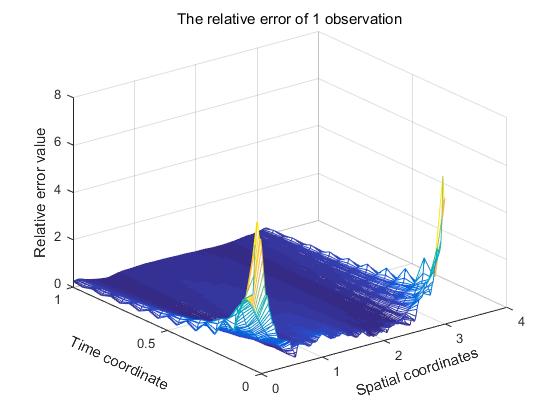}} \\
		\subfigure[$t=1$, N observations]{\includegraphics[width=0.29\textwidth]{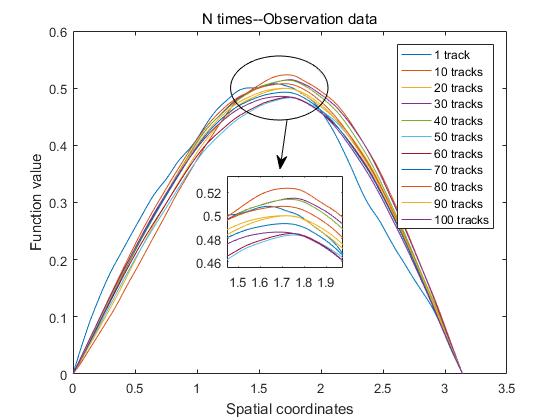}}	
		\subfigure[$t=1$, absolute error]{\includegraphics[width=0.29\textwidth]{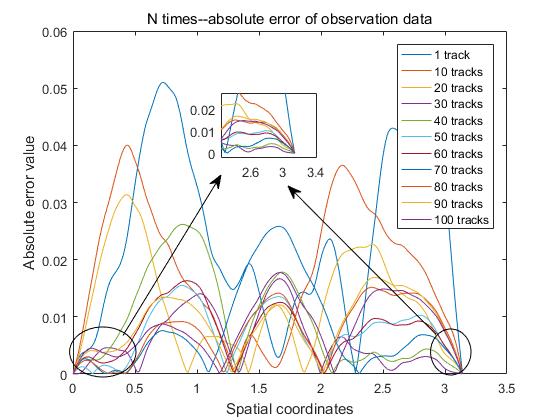}}	
		\subfigure[$t=1$, relative error]{\includegraphics[width=0.29\textwidth]{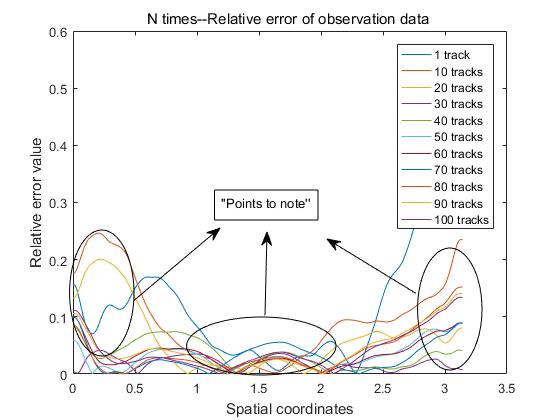}}	
		\caption{Information of the mild solution to Eq. (\ref{eq:6.2}).}
		\label{fig:6}
	\end{figure}
	\par
	Compared with Eq. (\ref{eq:6.1}), Eq. (\ref{eq:6.2}) differs in that its noise function is constant. Although this noise function is independent of space and time, Figure 6 (b) shows that the solution of Eq. (\ref{eq:6.2}) follows different distributions. As $x \to \partial D$, $\operatorname{var}(u(x, T)) \to 0$, making the observed data's covariance matrix singular. By comparing Figures \ref{fig:6} (b), (c), (e), and (f), we observe that the results are similar to Experiment 1; specifically, the covariance of the observed data and the relative error exhibit different fluctuation trends. This indicates that, although the noise function is constant in the SPDEs, the properties of the solution of SPDEs remain complex. Next, we employ the proposed method in the paper to address the stochastic inverse problem of Eq. (\ref{eq:6.2}).
	\begin{figure}[H]
		\centering
		\subfigure[10 track-inversion]{\includegraphics[width=0.29\textwidth]{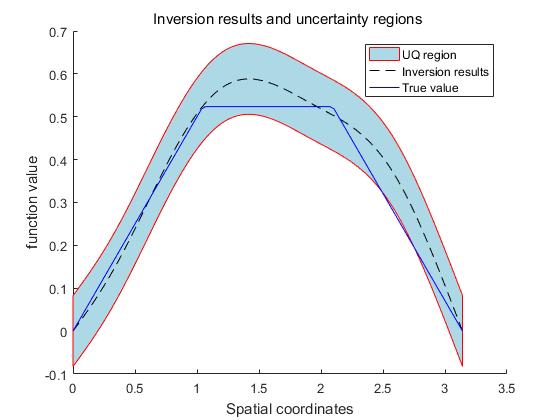}}
		\subfigure[20 track-inversion]{\includegraphics[width=0.29\textwidth]{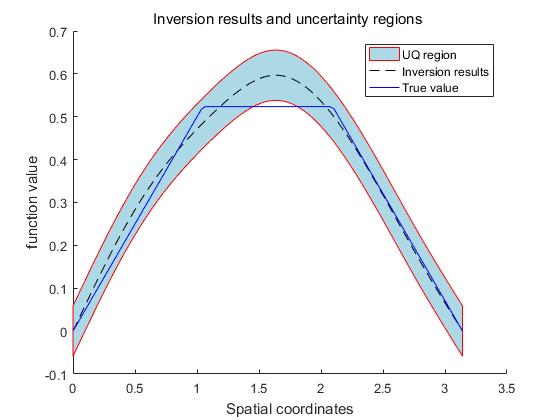}}
		\subfigure[40 track-inversion]{\includegraphics[width=0.29\textwidth]{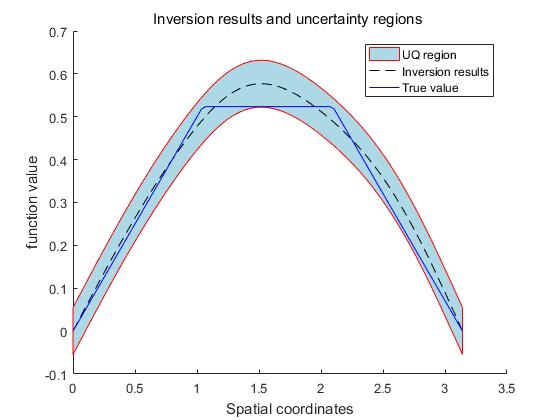}} \\
		\subfigure[100 track-inversion]{\includegraphics[width=0.29\textwidth]{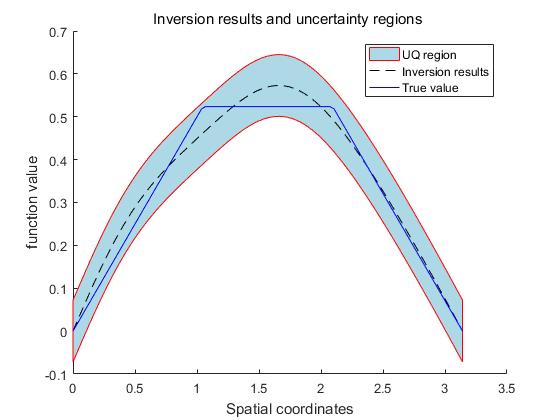}}	
		\subfigure[300 track-inversion]{\includegraphics[width=0.29\textwidth]{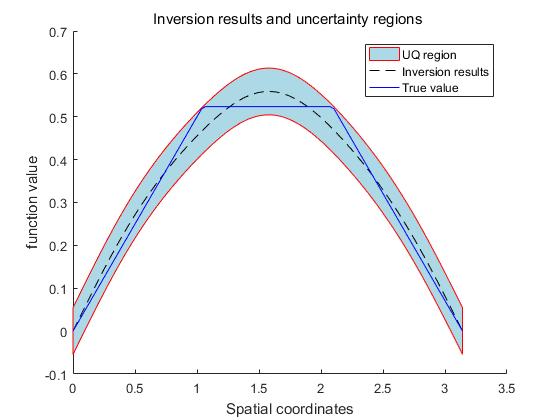}}	
		\subfigure[expectation-inversion]{\includegraphics[width=0.29\textwidth]{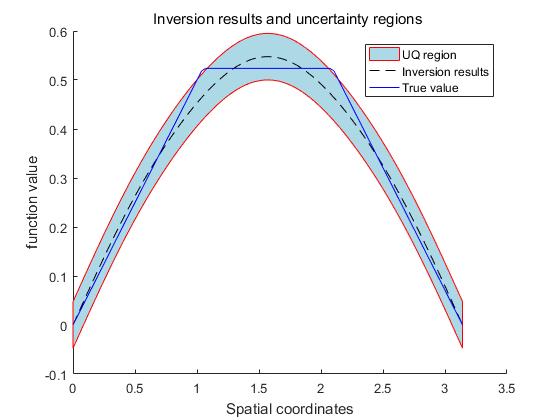}}	
		\caption{Comparison of inversion results under the condition of 0\% noise at time $t=1$.}
		\label{fig:7}
	\end{figure}
	\begin{figure}[H]
		\centering
		\subfigure[10 track-inversion]{\includegraphics[width=0.29\textwidth]{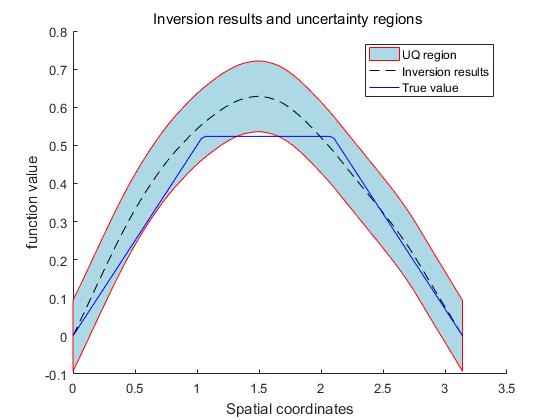}}
		\subfigure[20 track-inversion]{\includegraphics[width=0.29\textwidth]{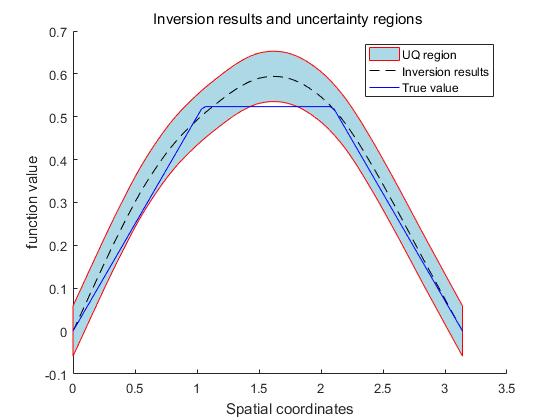}}
		\subfigure[40 track-inversion]{\includegraphics[width=0.29\textwidth]{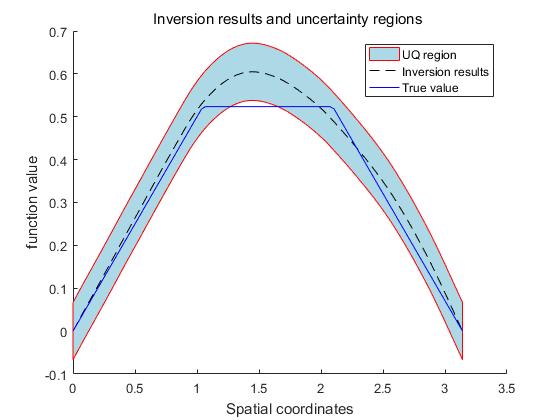}} \\
		\subfigure[100 track-inversion]{\includegraphics[width=0.29\textwidth]{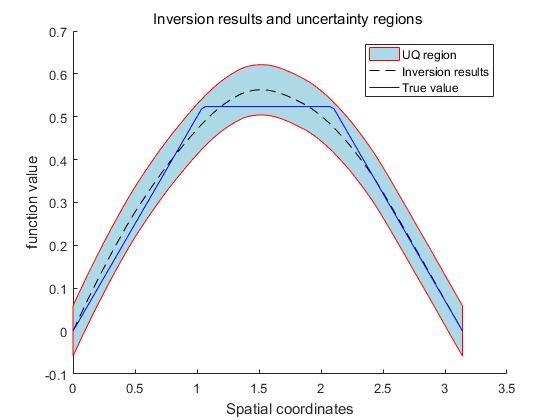}}	
		\subfigure[300 track-inversion]{\includegraphics[width=0.29\textwidth]{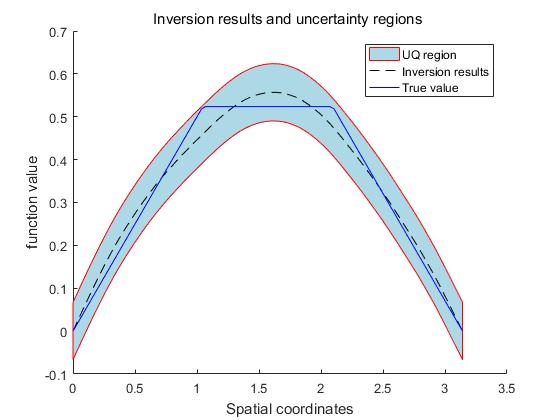}}	
		\subfigure[expectation-inversion]{\includegraphics[width=0.3\textwidth]{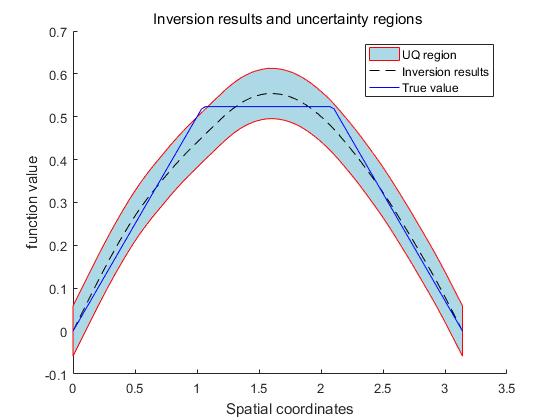}}	
		\caption{Comparison of inversion results under the condition of 1\% noise at time $t=1$.}
		\label{fig:8}
	\end{figure}
	\begin{figure}
		\centering
		\subfigure[10 track-inversion]{\includegraphics[width=0.29\textwidth]{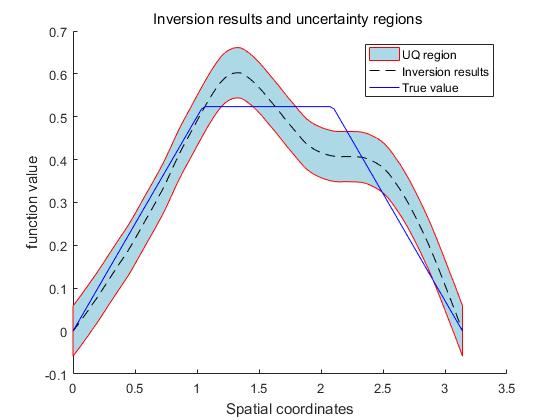}}
		\subfigure[20 track-inversion]{\includegraphics[width=0.29\textwidth]{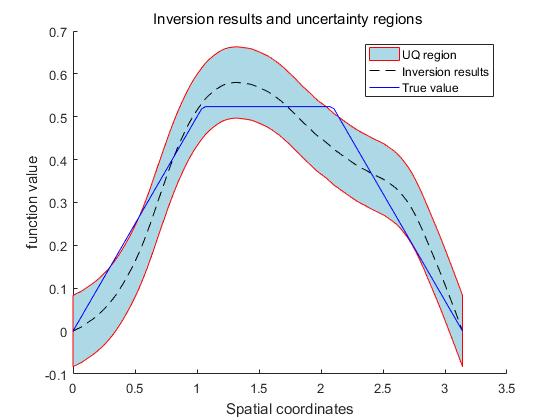}}
		\subfigure[40 track-inversion]{\includegraphics[width=0.29\textwidth]{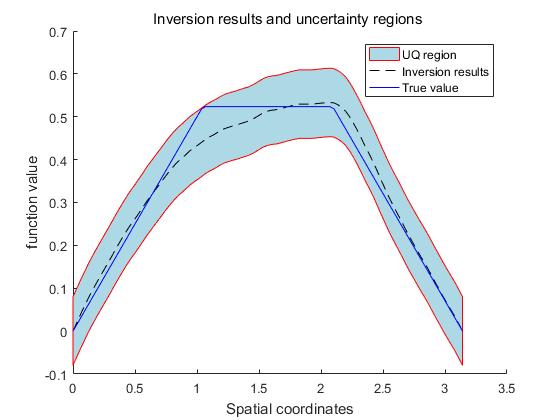}} \\
		\subfigure[100 track-inversion]{\includegraphics[width=0.29\textwidth]{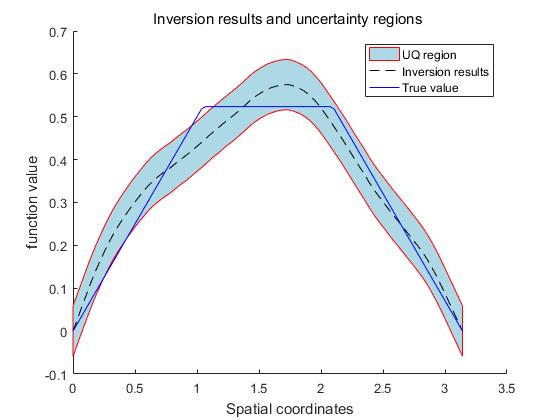}}	
		\subfigure[5\%expectation-inversion]{\includegraphics[width=0.29\textwidth]{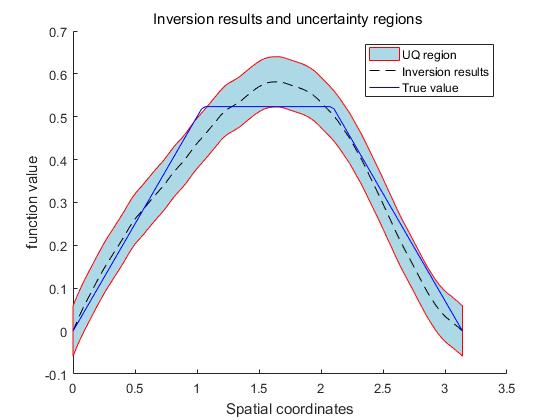}}
		\subfigure[10\%expectation-inversion]{\includegraphics[width=0.29\textwidth]{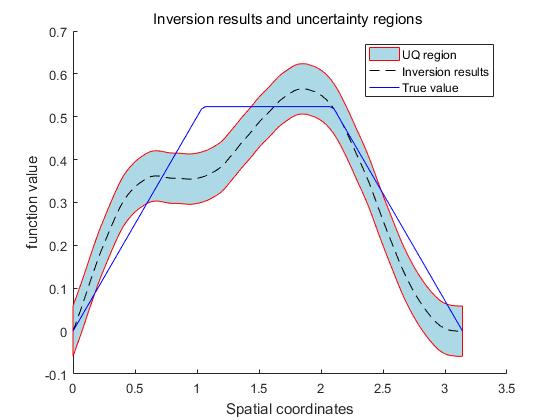}}	
		\caption{Comparison of inversion results under 5\% and 10\% noise conditions at time $t=1$.}
		\label{fig:9}
	\end{figure}
    \begin{figure}[H]
    	\centering
    	\subfigure[0\%- different track-inversion]{\includegraphics[width=0.33\textwidth]{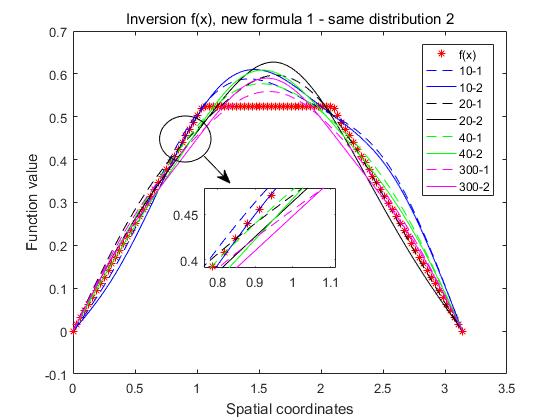}}
    	\subfigure[0\%-absolute error of inversion]{\includegraphics[width=0.33\textwidth]{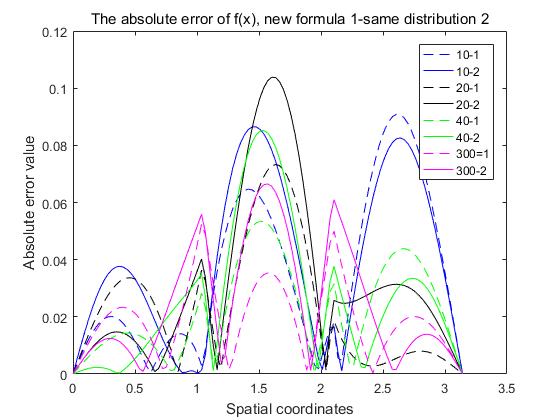}} \\
    	\subfigure[E-inversion]{\includegraphics[width=0.33\textwidth]{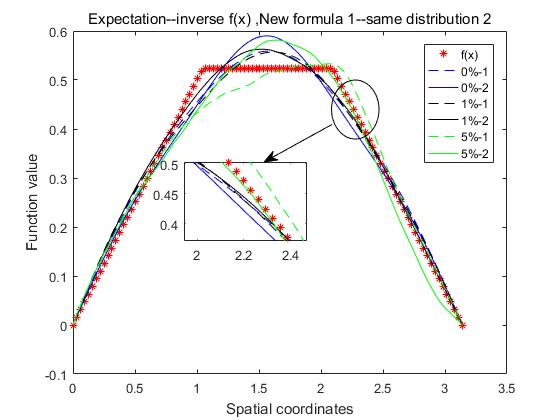}} 
    	\subfigure[E-absolute error of inversion]{\includegraphics[width=0.33\textwidth]{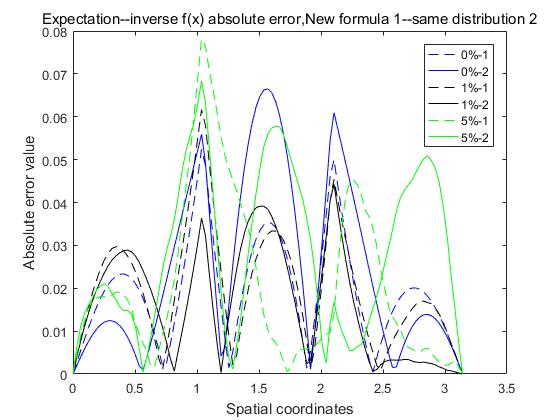}}	
    	\caption{Comparison of inversion results of two kinds of formulas}
    	\label{fig:10}
    \end{figure}
    \par
    In Figures \ref{fig:7} to \ref{fig:10}, we present the results of source term inversion experiments conducted under varying levels of unknown noise (0\%, 1\%, 5\%, and 10\%) and compare the reconstruction results of the formula proposed in this paper with the i.i.d. formula. Experiment 2 aligns well with the analytical results from Experiment 1, further confirming the validity of the proposed method. Thus, we will not reiterate this part. Detailed numerical variations are provided in Tables \ref{tab:5} and \ref{tab:6}. Next, we focus on the differences between Experiment 1 and Experiment 2.
    \begin{table}[H]
    	\begin{center}
    		\begin{tabular}{c|ccc}\toprule
    			Observation (N) & 10 track  & 20 tracks & 40 tracks \\
    			\hline
    			$\left \| f \right \|_{\infty}$ & $0.1038$ & $0.0897$ & $0.0637$ \\
    			$UQ_{max}$ & $0.1261$ & $0.0893$ & $0.0654$  \\
    			$UQ_{min}$ & $0.1256$ & $0.0884$ & $0.0653$  \\
    			\hline
    			Observation (N) & 100 tracks & 300 tracks & Expectation \\
    			\hline
    			$\left \| f \right \|_{\infty}$ & $0.0617$ & $0.0524$ & $0.0462$ \\
    			$UQ_{max}$ & $0.0624$ & $0.0536$  & $0.0475$ \\
    			$UQ_{min}$ & $0.0623$ & $0.0525$  & $0.0474$ \\
    			\bottomrule
    		\end{tabular}
    		\caption{The comparison of inversion error and uncertainty width of Eq. (\ref{eq:6.2}) under 0\% noise.}
    		\label{tab:5}
    	\end{center}
    \end{table}
    \begin{table}[H]
    	\begin{center}
    		\begin{tabular}{c|ccccc}
    			\toprule
    			\multicolumn{6}{c}{\textbf{Comparison of inversion results of $\left \| f \right \|_{\infty}$ under different noise levels, ($1\times 10^{-1}$)}}\\
    			\hline
    			Observation (N) & 10 track  & 20 tracks & 40 tracks & 100 tracks & 300 tracks \\
    			\hline
    			$0\%$ & $1.038$ & $0.838$ & $0.637$ & $0.617$ & $0.524$\\
    			$1\%$ & $1.286$ & $0.948$ & $0.731$ & $0.632$ & $0.621$\\
    			$5\%$ & $1.812$ & $1.139$ & $0.939$ & $0.881$ & $0.713$\\
    			$10\%$& $--$ & $--$ & $--$ & $--$ & $2.213$\\
    			\hline
    			\multicolumn{6}{c}{\textbf{Uncertainty width ($UQ_{min}$, $UQ_{max}$) under different noise levels, ($1\times 10^{-1}$)}}\\
    			\hline
    			Observation (N) & 10 track  & 20 tracks & 40 tracks & 100 tracks & 300 tracks \\
    			\hline
    			$0\%$ & $(1.256,1.261)$ & $(0.884,0.893)$ & $(0.653,0.654)$ & $(0.623,0.624)$ & $(0.541,0.542)$\\
    			$1\%$ & $(1.282,1.291)$ & $(0.944,0.948)$ & $(0.726,0.729)$ & $(0.635,0.638)$ & $(0.621,0.623)$\\
    			$5\%$ & $(1.445,1.473)$ & $(1.135,1.403)$ & $(0.911,0.916)$ & $(0.874,0.878)$ & $(0.712,0.714)$\\
    			\bottomrule
    		\end{tabular}
    		\caption{Compare the inversion error and uncertainty width of Eq. (\ref{eq:6.2}) under different noise conditions.}
    		\label{tab:6}
    	\end{center}
    \end{table}
    \par
	 Compared with Experiment 1, the increased complexity of the noise (see Figures \ref{fig:6} (c) and (e)) in Experiment 6.2 necessitates at least 10 observations to reconstruct the basic framework. However, no significant improvements were observed when observations increased from 40 to 100. This is primarily because the stochastic noise function $g(x)$ is constant, lacking distinct growth characteristics. Thus, more observed data do not significantly enhance the inversion capability. Furthermore, we observe that the inversion results in the middle region $\left [ \frac{\pi}{3},\frac{2\pi}{3} \right ]$ are less satisfactory. Even with 300 observations, the inversion quality has not improved substantially, particularly at non-differentiable points. This leads to a reduced capability for the UQ. This limitation arises from the Gaussian prior distribution used in this study, where the regularization term $\left \| f \right \| _{X}$ restricts the model’s ability to capture information at non-differentiable points. For instance, the Total Variation ($TV$) regularization prior method can help address this issue (for further analysis on alternative prior models, see \cite{ref34}). Finally, it is worth pointing out that the experimental results demonstrating the efficiency of the Bayesian variational inference method over the Monte Carlo sampling method for the UQ are detailed in \cite{ref16}. We do not discuss this further here.
	\section{Conclusions}
	We propose an efficient two-stage optimization method to address the instability and high computational cost associated with the Bayesian inversion method in stochastic inverse problems. Numerical experiments demonstrate that our proposed method can stably solve the Bayesian MAP estimation and efficiently quantify the uncertainty of the solution.
	\par
	The main contributions of this paper are as follows:
	\begin{itemize}
		\item {\bf New weighting formula (Q1):} We design a new weighting formula to effectively address the instability of Bayesian MAP estimation in stochastic inverse problems. Our analysis demonstrates that this formula converges to the standard weighting formula for i.i.d. observed data in the limit case, thereby establishing a theoretical connection between the study of inverse problems involving i.i.d. data and independent non-identically distributed data.
		\item {\bf Uncertainty of the optimization condition (Q2):} By combining the new weighting formula with the variational inference method, we derive the optimal condition for quantifying the solution's uncertainty. This condition efficiently reduces the high computational cost associated with Bayesian inversion methods for uncertainty quantification.
		\item {\bf Error estimation theory (Q3):} We establish an error estimation theorem relating the exact solution and the regularized solution under different amounts of the observed data using the eigensystem and the regularization method. This theorem provides a theoretical foundation for quantifying the uncertainty in the solution.
	\end{itemize}
	\par
	We believe the proposed method also applies to optimal estimation and UQ in other stochastic inverse problems, offering a novel approach to investigating these problems. Notably, while this work is framed within a Bayesian context, its essence is based on optimization methods. Therefore, we will focus on extending variational Bayesian methods to other stochastic inverse problems and exploring their applications in machine learning. Specifically, (1) leverage the optimization framework to develop more suitable priors and efficient sampling algorithms; (2) integrate the framework of this work with deep learning to advance data-driven research on stochastic inverse problems. These directions will form the basis of our ongoing and future efforts.
	\section{Acknowledgments}
	The authors would like to express their appreciation to the referees for their useful comments and to the editors for their support. Liying Zhang is supported by the National Natural Science Foundation of China (No. 11601514 and No. 11971458), the Fundamental Research Funds for the Central Universities (No. 2023ZKPYL02 and No. 2023JCCXLX01) and the Yueqi Youth Scholar Research Funds for the China University of Mining and Technology-Beijing (No. 2020YQLX03).
	
	\appendix
	\section{Appendix}
	\subsection{Proof of Theorem 2.1}
	\label{A1-section}
	\begin{proof}
		Based on (\ref{eq:2.3}), we consider the regular estimation of the mild solution in the expectation sense.
		\begin{align*}
			&E\left[\|u(x, t)\|_{L^{2}(D \times[0, T)}^{2}\right]  =E\left[\int_{0}^{T}\|u(x, t)\|_{L^{2}(D)}^{2} d t\right] \\
			& =E\left[\int_{0}^{T}\left\|\sum_{n=1}^{\infty}\left(e^{-\lambda_{n} t} u_{0, n}+\int_{0}^{t} e^{-\lambda_{n}(t-s)} R(s) f_{n} d s+\int_{0}^{t} e^{-\lambda_{n}(t-s)} g_{n} d w\right) \varphi_{n}(x)\right\|_{L^{2}(D)}^{2} d t\right] \\
			& =E\left[\int_{0}^{T}\left(\sum_{n=1}^{\infty}\left(e^{-\lambda_{n} t} u_{0, n}+\int_{0}^{t} e^{-\lambda_{n}(t-s)} R(s) f_{n} d s+\int_{0}^{t} e^{-\lambda_{n}(t-s)} g_{n} d w\right)^{2}\right) d t\right] \\
			& \leq 3 E\left[\int_{0}^{T} \sum_{n=1}^{\infty}\left(e^{-\lambda_{n} t} u_{0, n}\right)^{2} d t\right]+3 E\left[\int_{0}^{T} \sum_{n=1}^{\infty}\left(\int_{0}^{t} e^{-\lambda_{n}(t-s)} R(s) f_{n} d s\right)^{2} d t\right] \\
			& +3 E\left[\int_{0}^{T} \sum_{n=1}^{\infty}\left(\int_{0}^{t} e^{-\lambda_{n}(t-s)} g_{n} d w\right)^{2} d t\right].
		\end{align*}
		Consequently, for the first term, we have
		\begin{align*}
			E\left[\int_{0}^{T} \sum_{n=1}^{\infty}\left(e^{-\lambda_{n}, t} u_{0, n}\right)^{2} d t\right] & =\int_{0}^{T} \sum_{n=1}^{\infty}\left(e^{-\lambda_{n} t} u_{0, n}\right)^{2} d t \\
			& =\sum_{n=1}^{\infty}\left(u_{0, n}\right)^{2} \int_{0}^{T} e^{-2 \lambda_{n} t} d t =\sum_{n=1}^{\infty} \frac{1-e^{-2 \lambda_{n} T}}{2 \lambda_{n}}\left(u_{0, n}\right)^{2} \\
			& \leq T \sum_{n=1}^{\infty}\left(u_{0, n}\right)^{2} =T\left\|u_{0}(x)\right\|_{L^{2}(D)}^{2}.
		\end{align*}
		Similarly, for the second term, we compute as follows
		\begin{align*}
			E & \left[\int_{0}^{T} \sum_{n=1}^{\infty}\left(\int_{0}^{t} e^{-\lambda_{n}(t-s)} R(s) f_{n} d s\right)^{2} d t\right] =\int_{0}^{T} \sum_{n=1}^{\infty}\left(\int_{0}^{t} e^{-\lambda_{n}(t-s)} R(s) f_{n} d s\right)^{2} d t \\
			& =\sum_{n=1}^{\infty} \int_{0}^{T}\left(f_{n}\right)^{2}\left(\int_{0}^{t} e^{-\lambda_{n}(t-s)} R(s) d s\right)^{2} d t \leq C \sum_{n=1}^{\infty} \int_{0}^{T}\left(f_{n}\right)^{2}\left(\int_{0}^{t} e^{-\lambda_{n}(t-s)} d s\right)^{2} d t \\
			& \leq \frac{C T^{3}}{3} \sum_{n=1}^{\infty}\left(f_{n}\right)^{2} =\frac{C T^{3}}{3}\|f(x)\|_{L^{2}(D)}^{2}.
		\end{align*}
		Finally, we consider the third term, it can be derived by Ito's isometry formula, satifies
		\begin{align*}
			E\left[\int_{0}^{T} \sum_{n=1}^{\infty}\left(\int_{0}^{t} e^{-\lambda_{n}(t-s)} g_{n} d w\right)^{2} d t\right] & =\sum_{n=1}^{\infty}\left[\int_{0}^{T} E\left(\int_{0}^{t} e^{-\lambda_{n}(t-s)} g_{n} d w\right)^{2} d t\right] \\
			& =\sum_{n=1}^{\infty}\left[\int_{0}^{T} \int_{0}^{t}\left(e^{-\lambda_{n}(t-s)} g_{n}\right)^{2} d s d t\right] \\
			& =\sum_{n=1}^{\infty}\left(g_{n}\right)^{2}\left[\int_{0}^{T} \int_{0}^{t} e^{-2 \lambda_{n}(t-s)} d s d t\right] \\
			& =\sum_{n=1}^{\infty}\left(g_{n}\right)^{2}\left[\int_{0}^{T} \frac{1-e^{-2 \lambda_{n} t}}{2 \lambda_{n}} d t\right] \leq \sum_{n=1}^{\infty}\left(g_{n}\right)^{2}\left[\int_{0}^{T} t d t\right] \\
			& =\frac{T^{2}}{2} \sum_{n=1}^{\infty}\left(g_{n}\right)^{2} =\frac{T^{2}}{2}\|g(x)\|_{L^{2}(D)}^{2}.
		\end{align*}
		Thus, we obtain the regular estimation of the mild solution in the expectation sense following 
		\begin{align*}
			E \left[\|u(x, t)\|_{L^{2}(D \times[0, T)}^{2}\right] =E\left[\int_{0}^{T}\|u(x, t)\|_{L^{2}(D)}^{2} d t\right] \leq 3 T\left\|u_{0}(x)\right\|_{L^{2}(D)}^{2}+C T^{3}\|f(x)\|_{L^{2}(D)}^{2}+\frac{3 T^{2}}{2}\|g(x)\|_{L^{2}(D)}^{2}.
		\end{align*}
	\end{proof}
	

\end{document}